\documentclass[11pt]{article}
\usepackage{graphicx} % Required for inserting images

\usepackage{fullpage}

\usepackage{mathUtil}

\usepackage{todonotes}

\usepackage[backend=biber,style=alphabetic,maxnames=50]{biblatex}
\usepackage{authblk}

\usepackage{hyperref,comment}
\hypersetup{
    colorlinks=true,
    linkcolor=blue,
    filecolor=magenta,
    urlcolor=cyan,
    pdftitle={Locally Esakia Spaces and Esakia order-compactifications},
    pdfpagemode=FullScreen,
}

\title{Esakia order-compactifications and locally Esakia spaces}

\author{Rodrigo Nicolau Almeida, Guram Bezhanishvili, and Nick Bezhanishvili}

\begin{document}
\date{}
\maketitle

\begin{abstract}
We introduce 
%Heyting and 
Esakia order-compactifications 
%of ordered topological spaces, 
and study how they fit in the general theory of %Nachbin and 
Priestley order-compactifications. We provide an analog of Dwinger's theorem by  %connecting 
characterizing Esakia order-compactifications by means of
%with 
special rings of upsets. These considerations naturally lead to 
%also introduce 
the notion of a locally Esakia space, for which we prove that taking the largest Esakia order-compacification is functorial, thus obtaining an analog of Banaschewski's theorem.
%, and show that taking the largest
%greatest 
%Esakia order-compacification is functorial \color{red} doesn't this repeat the same idea, given that Banaschewski's result concerns the largest Esakia order-compactification\color{black}.   
\end{abstract}

\tableofcontents

\vspace{5mm}


\vspace{1mm}


\section{Introduction}

The theory of compactifications is one of the mainstream areas of
%plays an important 
%a key 
%role in 
point-set topology. It is a classic result 
%of Tychonoff 
that a topological space $X$ admits a compactification iff it is completely regular, and that the Stone-\v{C}ech compactification $\beta X$ is the largest 
%greatest element of the poset of 
compactification of $X$ (see, e.g., \cite[Ch.~3.5]{Engelking1989-qo}). 
%Notably, 
In fact, the correspondence $X\mapsto\beta X$ extends to
%taking the Stone-\v{C}ech compactification defines 
a functor which is left adjoint to the inclusion of the category of compact Hausdorff spaces into the category of completely regular spaces 
%\cite{Cech1937} 
(see, e.g., \cite[pp.130-131]{johnstone1982stone}). 

In algebraic logic, a special role is played
%occupied 
by those compact Hausdorff spaces that are zero-dimensional as they serve as Stone duals of Boolean algebras, and are widely known as {\em Stone spaces}. %\cite{Stone1934}. which , known as Stone spaces,  --  -- play a crucial role as the duals of Boolean algebras , 
This motivates the study of zero-dimensional compactifications, the theory of which resembles that of all compactifications: a topological space $X$ admits a zero-dimensional compactification iff it is a zero-dimensional (Hausdorff) space, and there exists the largest zero-dimensional compactification $\beta_0 X$, which is the Stone dual of the Boolean algebra $\mathsf{Clop}(X)$ of all clopen susbets of $X$ \cite{Banaschewski1955}. The compactification $\beta_0 X$ is known as the {\em Banaschewski compactification} and the correspondence $X \mapsto \beta_0 X$ extends to a functor which is left adjoint to the inclusion of the category of Stone spaces into the category of zero-dimensional spaces \cite{Banaschewski1955} (see also \cite{Banaschewski1963}). 
By Dwinger's theorem \cite[Thm.~13.1]{dwinger1961} (see also \cite{Magill1968}), all zero-dimensional compactifications of $X$ can be characterized as Stone duals of {\em Boolean bases}; that is, Boolean subalgebras of $\mathsf{Clop}(X)$ that are bases for the topology.   

%are of special importance: they are equivalent to specific rings of sets \cite{Magill1968}, they provide dual representations of products of Boolean algebras \cite{Dwinger1960}, and they allow describing different field representations of Boolean algebras \cite{Bauer1955}. Of special interest amongst these is the Banachewski compactification, which similarly to the Stone-\v{C}ech compactification,  provides a left adjoint to the inclusion of the category of zero-dimensional spaces into the category of Stone spaces.

% A classic result by Dwinger establishes a one-to-one correspondence between zero-dimensional compactifications and Boolean bases of a zero-dimensional space [REF].

% \color{blue} This should be mentoned right after the first sentence about compactifications. Zero-dimensional compactifications can come second. \color{black}
% It is also an crucial fact that taking the Stone–Cech compactification is . Likewise, taking the Banaschewski compactification is a left adjoint to the inclusion of the category of Stone spaces into the category of Hausdorff zero-dimensional spaces \cite{Banaschewski1955}. 

% \color{blue} Do we need this general sentence (especially since it gives very few references; there's a whole bunch more that could be given)? How about we directly transition to Priestley and mention that the above generalizes to that setting? The idea being that we are interested in spaces associated with boolean algebras, DLs, and HAs. \color{black}

Nachbin generalized the theory of compactifications to that of 
%The theory of 
order-compactifications (see, e.g., \cite{nachbin1965topology}), which was further developed 
%studied 
by numerous authors. 
%of order-topological spaces has also been thoroughly studied \cite{nachbin1965topology,BLATTER197556,Richmond1993,Nailana2000}. 
In this theory,
%Among order-compactifications, 
the role of zero-dimensional compactifications is played by Priestley order-compactifications \cite{Koubek1991,BezhanishviliMorandi2010}. Priestey spaces are of special importance in algebraic logic as they serve as Priestley duals of distributive lattices, thus generalizing Stone duality for Boolean algebras. 
%important class of order-topological spaces is given by Priestley spaces, compact spaces which are order-zero-dimensional. It was shown by Priestley \cite{Priestley1970} that the category of Priestley spaces and continuous order-preserving  maps is dual to the category of distributive lattices and lattice homomorphism, generalizing Stone duality. 
For ordered spaces, the analog of zero-dimensionality is that of order-zero-dimensionality, and an ordered space $X$ admits a Priestley order-compactification iff it is order-zero-dimensional \cite[Cor.~3.8]{BezhanishviliMorandi2010}. Moreover, there exists the largest Priestley order-compactification $\eta_0 X$, which is the Priestley dual of the distributive lattice $\mathsf{ClopUp}(X)$ of all clopen upsets of $X$, and the correspondence $X \mapsto \eta_0 X$ extends to a functor which is left adjoint to the inclusion of the category of Priestley spaces into the category of order-zero-dimensional spaces \cite[Rem.~3.10]{BezhanishviliMorandi2010}. Furthermore, all Priestley order-compactifications of $X$ can be characterized as Priestley duals of {\em Priestley bases}; that is, sublattices of $\mathsf{ClopUp}(X)$ that %give rise to bases for 
generate the topology of $X$  \cite[Sec.~5]{BezhanishviliMorandi2010}. Thus, both Banschewski's compactification  and Dwinger's theorem have natural generalizations to the order-zero-dimensional setting. 

%has been developed in  (see also \cite{Koubek1991}), where it was shown that there is a one-to-one correspondence between the poset of Priestley order-compactifications of an order-zero-dimensional space $X$ and the poset of Priestley bases (special rings of upsets) of $X$. 
%Moreover, the greatest Priestley compactification, $\eta_{0}$, provides an order-topological analogue of Banachewski's result: as a functor, it is the left adjoint to the inclusion of Priestley spaces in the category of order-zero-dimensional spaces (see Theorem \ref{cor: Adjunction for inclusions in Order-zero-dimensional and Priestley}).

% \color{blue} This is not true, you need an analogue of strong zero-dimesnionality for Nachbin to be Priestley. \color{black} It was shown in [Ref] that the Nachbin compactification provides a left adjoint of the inclusion of Prietley spaces into ... spaces. \color{blue} Not Nachbin but what we denote by $\eta_0$. \color{black} 

Alongside Stone spaces and Priestley spaces, Esakia spaces play a crucial role in algebraic logic as they serve as Priestley duals of Heyting algebras (see, e.g., \cite{Esakiach2019HeyAlg}). 
%Next to Priestley spaces, Esakia spaces play a crucial role in the theory of order-topological dualities. These are 
Esakia spaces are those Priestley spaces in which the order is continuous. 
%where the partial order is in addition continuous \cite{Esakiach2019HeyAlg}. The category of Esakia spaces and continuous p-morphisms is dual to the category of Heyting algebras and Heyting algebra homomorphisms \cite{esakiatopologicalkripke,Esakiach2019HeyAlg}. From a logical point of view, Esakia spaces provide a sound and complete order-topological semantics for all intermediate logics \cite{Chagrov1997-cr}.
%
%\rodrigonote{Should we have this sentence, since it refers to logics, which don't really appear anywhere else?}
%
% \color{blue} I think we need to explain things better in this paragraph. We need to talk about how we refine the notion of order-zero-diemnsionality to E-order-zero-dimensionality. We also need to explain the importance of p-morphisms, which is responsible for our new notion of image-compact, which needs to be explained. After that we can transition to locally Esakia which is the conjunction of the two, and point out that it is exactly this category that provides the right analog for us of the category of zero-dimensional Hausdorff spaces. \color{black}
Our aim is to develop the theory of %Priestley order-compactifications develop an analogous theory of 
Esakia order-compactifications that refines that of Priestley order-compactifications, and provide a generalization of the Banaschewski and Dwinger results in the setting of continuous orders. As we will see, this gives rise to two non-equivalent notions of order-compactifications, which we term Heyting and Esakia order-compactifications (see Definition \ref{def: Heyting and Esakia compactifications}).  These, in turn, induce two refinements of order-zero-dimensionality: H-order-zero-dimensionality and E-order-zero-dimensionality (see Definition \ref{def: H and E OZD}), 
%(H for Heyting and E for Esakia), 
the latter being especially pertinent for our purposes. 
%One of these, which we term E-order-zero-dimensionality, is especially relevant for the study of 

%In this paper we extend the theory of Priestley order-compactifications to the setting of Esakia spaces. More concretely, we define two notions, Heyting order-compactifications and Esakia order-compactifications, both with respect to order-zero-dimensional spaces. A Heyting order-compac\-ti\-fi\-cation is a Priestley order-compactification $Y$ such that $Y$ is an Esakia space. In Esakia order-compac\-ti\-fi\-cations we additionally require that the upset of each point $x$ in the original space is dense in the upset of $x$ in the compactification. Alongside these we  introduce the notions of $H$-order-zero-dimensionality and $E$-order-zero-dimensionality, corresponding to admitting respectively a Heyting and an Esakia order-compactification. $E$-order-zero-dimensionality has many natural characteristics of Esakia spaces, for example, it is equivalent to continuous orderability -- a property which may fail for $H$-order-zero-dimensional. We also provide an analogue of Dwinger’s theorem for both Heyting and Esakia order-compactifications in terms of Heyting rings of upsets and Esakia rings of upsets.

We provide an analog of Dwinger’s theorem for both Heyting and Esakia order-compactificati\-ons, and then consider the problem of functoriality of Esakia order-compactifications. Since the maps in the category of Esakia spaces are continuous p-morphisms, it is natural to consider them as the maps in the category of E-order-zero-dimensional spaces. This leads 
%We are thus lead 
to the key notion of a {\em locally Esakia space}, an E-order-zero-dimensional space in which the upset of each point is compact. This notion naturally generalizes both Esakia spaces and image-finite posets (in which the upset of each point is finite). We prove that the correspondence $X\mapsto \eta_0 X$ lifts to a functor which is left adjoint to the inclusion of the category of Esakia spaces into the category of locally Esakia spaces and continuous p-morphisms, thus providing an analog of Banaschewski's result for locally Esakia spaces.

We are happy to contribute this work to the volume dedicated to Mai Gehrke.
Mai has been a colleague and friend (for the second and third-named authors), and a central figure in the TACL community for decades.\footnote{TACL stands for Topology, Algebra and Categories in Logic \color{blue} \url{https://math.univ-cotedazur.fr/tacl/}. \color{black}}  Order-compactifications have played an important role in Mai’s work. For example, in her paper for the Esakia volume \cite{Gehrke2014},  
Mai connected order-compactifications of the 
$n$-universal models of intuitionistic logic to 
the bicompletions of quasi-uniform spaces. Her joint work \cite{Bezhanishvili2006} develops a theory of profinite completions which relies heavily on order-compactifications of image-finite posets. By introducing
locally Esakia spaces, we provide a natural generalizatation of this approach.

\section{Order-separation and order-zero-dimensionality}\label{Section: Preliminaries}

% \color{red} 2.1 and 2.2 can be sections on their own. \color{black}

We assume the reader's familiarity with lattice theory (see, e.g., \cite{balbes1974distributive}), topology (see, e.g., \cite{Engelking1989-qo}), and category theory (see, e.g., \cite{MacLane1978}). 
As in \cite{johnstone1982stone}, all lattices are assumed to be bounded and all lattice homomorphisms preserve the bounds. We will use letters $L,D$ to denote distributive lattices, $H$ to denote Heyting algebras, and $B$ to denote Boolean algebras.

As is customary, for a poset $X$ and $S\subseteq X$, we write 
\[
{\uparrow}S=\{x\in X \mid \exists s\in S : s\leq x\} \quad \mbox{and} \quad {\downarrow}S=\{x\in X \mid \exists s\in S : x\leq s\}.
\]
We say that $S$ is an \textit{upset} if $S={\uparrow}S$ and a {\em downset} if $S={\downarrow}S$. 
If $S = \{ x \}$ is a singleton, we simply write ${\uparrow}x$ and ${\downarrow}x$. We let $\mathsf{Up}(X)$ be the set of upsets of $X$, and note that $\mathsf{Up}(X)$ is a complete and completely distributive lattice, where joins and meets are set-theoretic union and intersection.

% \subsection{Order-separation and order-zero-dimensionality}

By an 
\textit{ordered topological space} or simply an {\em ordered space} we mean a poset $X$ which is also a topological space. 
The study of ordered topological spaces was pioneered by Nachbin. His main findings were collected in  %\color{blue} [REF]. \color{black}
%In this section we recall some basic facts about ordered-topological spaces. Most of the results of this section trace their origin to Nachbin's pioneering work 
\cite{nachbin1965topology}, which we use as one of the main references in this section.

\begin{definition}
    An ordered 
    %topological 
    space $X$ is a \textit{Nachbin space} if $X$ is compact and the order $\le$ is closed (in the product space $X\times X$).  
\end{definition}

It is well known (see \cite[p.~26]{nachbin1965topology}) that the order $\le$ being closed is equivalent to the following order-separation axiom: 
For each $x,y\in X$, if $x\nleq y$ then there exist an upset neighbourhood $U$ of $x$ and a downset neighbourhood $V$ of $y$ such that $U\cap V=\emptyset$. For obvious reasons, this separation is known as \textit{order-Hausdorffness} (see \cite{mccartan_1968}). In particular, it implies that $X$ is a Hausdorff space.

In this paper, we will mainly be interested in order-zero-dimensional spaces, which are order-topological analogues of zero-dimensional spaces. We recall that a subset of a topological space $X$ is {\em clopen} if it is both closed and open, and that $X$ is {\em zero-dimensional} if the set ${\sf Clop}(X)$ of clopen subsets of $X$ is a basis for the topology on $X$. For an ordered 
%topological 
space $X$, we let $\mathsf{ClopUp}(X)$ be the set of clopen upsets of $X$.
%, \color{red} and $\mathsf{ClopDown}(X)$ the set of clopen downsets\color{black}. 

%\color{blue} Should this be in the convention environment? \color{black}
%Throughout 
As a general rule, we will refer to 
%deal with properties of 
order-topological properties as 
%spaces. %Given a property $P$ we will refer to spaces as satisfying 
``order-P,"
 e.g., order-zero-dimensionality. Some authors call such spaces zero-dimensional ordered spaces (see, e.g., 
 %refer to this property as and completely-order-regular); this stands in contrast with some authors (e.g. see 
 \cite{Nailana2000}). Since the latter could alternatively be interpreted as an ordered space that happens to be zero-dimensional, we prefer the usage of ``order-P" to avoid any unnecessary confusion.
 %who will, for example, refer to ``completely regular ordered space"; we find that this might lead to confusion, and stick to our convention, since completely regular ordered spaces are not simply completely regular spaces which happen to be ordered.

\begin{definition}
   Let $X$ be an ordered 
   %topological 
   space.
   \begin{enumerate}
       \item (Priestley separation) We say that $X$ satisfies the {\em Priestley separation axiom} provided for each $x,y\in X$, if $x\not\le y$ then there is a clopen upset $U$ such that $x \in U$ and $y \notin U$.
       \item (Order-zero-dimensionality) We call $X$ \textit{order-zero-dimensional} if 
       \begin{enumerate}
           \item $X$ satisfies the Priestley separation axiom, and 
           \item the set $\{U \, \setminus \, V : U,V\in \mathsf{ClopUp}(X)\}$ is a basis for the topology.
       \end{enumerate}
   \end{enumerate} 
\end{definition}

\begin{remark}
It is clear that Priestley separation implies order-Hausdorffness. Moreover, since $U\,\setminus\,V$ is a convex set for each $U,V\in \mathsf{ClopUp}(X)$, each order-zero-dimensional space has a basis of clopen convex sets, hence is locally convex (see \cite[p.~27]{nachbin1965topology}).
\end{remark}

A class of ordered spaces closely related to order-zero-dimensional spaces is that of Priestley spaces, which play a central role in the representation theory of distributive lattices.

\begin{definition}\cite[258]{Davey2002-lr}
    An ordered 
    %topological 
    space $X$ is a \textit{Priestley space} if it is compact and satisfies the Priestley separation axiom.
\end{definition}

\begin{theorem} [Priestley duality,  \cite{Priestley1970}]
The category $\sf Pries$ of Priestley spaces and continuous order-preser\-ving maps is dually equivalent to the category $\sf DLat$ of distributive lattices and lattice homomorphisms. 
\end{theorem}

Priestley 
%spaces were introduced in \cite{Priestley1970} with the goal of providing a 
duality 
%theory for distributive lattices 
closely parallels Stone duality for Boolean algebras. We briefly recall the action of the functors on objects as this will be used in what follows.  With each Priestley space $X$ we associate its dual distributive lattice $\mathsf{ClopUp}(X)$, and with each continuous order-preserving map $f:X\to Y$, 
%between Priestley spaces, 
the lattice homomorphism $
%f^{*}=
f^{-1}:\mathsf{ClopUp}(Y)\to \mathsf{ClopUp}(X)$. 

With each distributive lattice $D$, we associate the Priestley space $\mathsf{Spec}(D)$ of prime fiters of $D$ ordered by inclusion and topologized by the 
    %denote by $(\mathsf{Spec}(D),\subseteq,\tau)$ the \textit{Priestley dual} of $D$, where $\tau$ is given by a 
    %sub
    basis $\{\phi(a)
    %:a\in D\}\cup\{\mathsf{Spec}(D)
    \,\setminus\,\phi(b) : a,b\in D\}$,  where
    \begin{equation*}
        \phi(a)=\{P\in \mathsf{Spec}(D) : a\in P\}.
    \end{equation*}
With each lattice homomorphism $f:D\to D'$ we associate the continuous order-preserving map 
%\begin{equation*}
    %f_{*}\coloneqq 
    $f^{-1}:\mathsf{Spec}(D')\to \mathsf{Spec}(D)$.
%\end{equation*}

%\color{red}
%The action of the functors on morphisms is rather simple. If $f:D\to D'$ is a distributive lattice homomorphism, then we associate to it its dual map
%\begin{equation*}
%    f_{*}\coloneqq f^{-1}:\mathsf{Spec}(D')\to \mathsf{Spec}(D),
%\end{equation*}
%which will be a continuous order-preserving map. In turn with each continuous and order-preserving map $f:X\to Y$ between Priestley spaces, $f^{*}=f^{-1}:\mathsf{ClopUp}(Y)\to \mathsf{ClopUp}(X)$ is a distributive lattice homomorphism.
%\color{black}

%\begin{proposition}\label{One half of Priestley duality}
%    Given a Priestley space $X$, $X\cong \mathsf{Spec}(\mathsf{ClopUp}(X))$.
%\end{proposition}
%\begin{proof}
%    (Add reference)
%\end{proof}

We will freely use the following well-known facts about Priestley spaces (see, e.g.,  \cite[Prop.~2.6]{Priestley1984} and \cite[Ch.~11]{Davey2002-lr}): 
%\cite[Lem.~11.22]{Davey2002-lr} for \ref{eqref: Clopen upsets and downsets form a subbasis}):

\begin{fact} \label{fact: Basic facts about Priestley spaces}
    %The following facts about Priestley spaces are well-known:
    For each Priestley space $X$, we have:
    \begin{enumerate}
        \item $X$ is a Nachbin space. Consequently, if $F$ is a closed subset of $X$, then both ${\uparrow}F$ and ${\downarrow}F$ are closed.
        \color{black}
        \item \label{eq: Closed upsets can be separated} 
        If $F$ and $G$ are closed subsets of $X$ such that %is a closed upset, $B$ is a closed downset, and 
        ${\uparrow}F\cap{\downarrow}G=\emptyset$, then there is a clopen upset $U$ of $X$ such that $F\subseteq U$ and $G\subseteq X \,\setminus\, U$.
        \item \label{eqref: Clopen upsets and downsets form a subbasis} The set $\{ U\,\setminus\,V : U,V\in\mathsf{ClopUp}(X) \}
        $ 
        forms a basis for the topology on $X$. Consequently, each clopen $W \subseteq X$ is of the form $W=\bigcup_{i=1}^{n}\left(U_{i}\,\setminus\,V_{i}\right)$, where $U_{i},V_{i}$ are clopen upsets.
    \end{enumerate}
\end{fact}

\section{Order-compactifications and Priestley order-compactifications}

The notion of order-compactification has two natural generalizations to the setting of ordered %topological 
spaces. The more restrictive of the two was already studied by Nachbin \cite{nachbin1965topology}, which we call \textit{N-order-compactification}; the more general one was later introduced by Blatter \cite{BLATTER197556}, and became more standard. 

\begin{definition}
    A map $f:X\to Y$ between ordered 
    %topological 
    spaces is an \textit{order-embedding} if
    \begin{itemize}
        \item $f$ is a topological embedding;
        \item $x\leq y$ iff
        %if and only if 
        $f(x)\leq f(y)$.
    \end{itemize}
\end{definition}

\begin{definition}
    Let $X$ be an ordered 
    %topological 
    space.
    \begin{enumerate}
        \item An {\em order-compactification} of $X$ is a pair $(Y,e)$, where $Y$ is a Nachbin space and $e: X \to Y$ is an order-embedding such that $e[X]$ is dense in $Y$.
        \item We call an order-compactification $(Y,e)$ an {\em N-order-compactification} if the order $\le_Y$ on $Y$ is the closure of the image of $\le_X$ under the product map $e\times e : X \times X \to Y \times Y$.%\footnote{That is, the order on $X$ is dense in the order on $Y$.}
    \end{enumerate}
\end{definition}

\begin{remark} \label{rem: identification}
When it does not cause any confusion, we  identify $X$ with its image $e[X]$ in $Y$, view $X$ as a dense subspace of $Y$, and $Y$ as an order-compactification of $X$. Then, $Y$ is an N-order-compactification of $X$ iff the order on $X$ is dense in the order on $Y$.
%omitting the mention of the map $e$.
\end{remark}

It is a classic result that a topological space $X$ has an order-compactification iff it is completely regular, in which case the collection of  (equivalence classes of) all compactifications has a natural partial order and the Stone-\v{C}ech compactification $\beta X$ is the largest compactification in this partial order \cite[pp.166-168]{Engelking1989-qo}. These results generalize to the setting of order-compactifications \cite{BLATTER197556}.

%\color{blue} We need to agree about what to call things. Should we follow Nachbin in simply adding ``ordered'' to a topological property P, like completely regular ordered space, or rather say order-P? \color{black}

%\color{red}
%Rodrigo: I'm happy with either option!
%\color{black}

\begin{definition} \cite[Sec.~II.1]{nachbin1965topology}
    %Let $(X,\leq)$ be 
    An ordered 
    %topological 
    space $X$ is
    %\begin{enumerate}
    %    \item \textit{Order-Hausdorff} \footnote{This term was first introduced by McCartan in \cite{mccartan_1968}.} if for each $x,y\in X$, if $x\nleq y$, then there exists some upset neighbourhood $U$ of $x$ and a downset neighbourhood $V$ of $y$ such that $U\cap V=\emptyset$.
    %    \item 
    \textit{completely-order-regular}
    %\footnote{These were introduced in \cite[pp.54]{nachbin1965topology} under the name ``completely regular ordered spaces". We follow \cite{BezhanishviliMorandi2010} in this naming convention.} 
    if
        \begin{enumerate}
            \item For each $x,y\in X$, if $x\nleq y$ then there is a continuous order-preserving %function 
            $f:X\to [0,1]$ such that $f(x)>f(y)$.
            \item For each $x\in X$ and  closed set $F$ with $x\notin F$, there exist a continuous order-preserving 
            %function 
            $f:X\to [0,1]$ and a continuous order-reversing %function 
            $g:X\to [0,1]$ such that $f(x)=1=g(x)$ and $F\subseteq f^{-1}(0)\cup g^{-1}(0)$.
        \end{enumerate}
    %\end{enumerate}
\end{definition}

%The notion of strong order-compactification was originally introduced by Nachbin, under the term of order-compactifications; Blatter later generalized this \todo{Add references}. 

\begin{theorem} {\em \cite[Thm.~1.5(i)]{BLATTER197556}}
    An ordered space has an order-compactification iff it is comple\-tely-order-regular.
\end{theorem}

%\color{red}It was shown by Blatter that an ordered space has an order-compactification iff it is completely-order-regular \cite[5.16]{Blatter1976} \color{black}. 

We also have an obvious order on order-compactifications:

\begin{definition}
    Let $(Y_{1},e_{1})$ and $(Y_{2},e_{2})$ be order-compactifications of a completely-order-regular %topological 
    space $X$. We write $(Y_{1},e_{1})\preceq (Y_{2},e_{2})$ if there exists a continuous order-preserving 
    %function 
    ${f:Y_{2}\to Y_{1}}$ such that $f \circ e_{2} = e_{1}$. 
\end{definition}

Since $e_1[X]$ is dense in $Y_1$, such an $f$ is always onto. Moreover, as for compactifications, we have that $(Y_{1},e_{1})\preceq (Y_{2},e_{2})$ and $(Y_{2},e_{2})\preceq (Y_{1},e_{1})$ iff there is an order-homeomorphism 
%between $Y_1$ and $Y_2$ 
$f:Y_1\to Y_2$ such that $f \circ e_{1} = e_{2}$, and this defines a partial order on the collection of (equivalence classes of) order-compactifications \cite[57]{BLATTER197556},
%\cite[5.16]{Blatter1976}
which we denote by $\mathcal{K}(X)$. 
%be the poset of (equivalence classes of) order-compactifications of $X$.
%\cite{Richmond1993}. 

The largest element of this poset 
%of order-compactifications of a completely-order-regular space 
was constructed by Nachbin \cite[App.~2]{nachbin1965topology} by generalizing the construction of the Stone-\v{C}ech compactificaion to the setting of ordered spaces. We refer to it as the {\em Nachbin order-compactification} and denote it by $\eta X$.\footnote{The Nachbin order-compactification is sometimes denoted by $\beta_0 X$, $\beta_1 X$ (see, e.g., \cite{choe1979,Richmond1993}), 
or $n X$ (see, e.g., \cite{Bezhanishvili2006,BezhanishviliMorandi2010}).
%to allow ease of reading, we introduce this notation.
} It is worth pointing out that $\eta X$ is in fact an N-order-compactification \cite[p.~103]{nachbin1965topology}.  Similar to the Stone-\v{C}ech compactification, the Nachbin order-compactification has the following universal mapping property:

\begin{theorem} {\em (see, e.g., \cite[p.~56]{choe1979})}
    If $Z$ is a Nachbin space, then each continuous order-preserving 
    %function 
    $f:X\to Z$ has a unique extension to a continuous order-preserving 
    %function 
    $g:\eta X\to Z$. 
    %such that $g\circ e=f$.
\end{theorem}

For various constructions of $\eta X$ we refer to \cite[pp.~56--57]{choe1979}. As with the Stone-\v{C}ech compactification, the construction of $\eta X$ is involved, but things simplify if $X$ is order-zero-dimensional.

%We now restrict our attention to those $X$  that are order-zero-dimensional.

\begin{definition}{\cite{Koubek1991,BezhanishviliMorandi2010}}
    We call an order-compactification $Y$ of $X$ a \textit{Priestley order-compactification} if $Y$ is a Priestley space (equivalently, $Y$ is order-zero-dimensional).
\end{definition}

We point out that $X$ has a Priestley order-compactification iff $X$ is order-zero-dimensional \cite[Thm.~3.5]{BezhanishviliMorandi2010}. 
%\color{red} Consequently, 
For such an $X$, let $\mathcal{K}_{0}(X)$ be the subposet of 
%the poset of 
$\mathcal{K}(X)$ consisting of Priestley order-compactifications. The poset $\mathcal{K}_{0}(X)$ can be described 
 %Priestley order-compactifications can be characterized 
in terms of special rings of sets studied in \cite{BezhanishviliMorandi2010}:

\begin{definition}\ \label{Priestley rings}
    \begin{enumerate}
        \item A \textit{ring of upsets} of a poset $X$ is a collection $\mathcal{R}\subseteq \mathsf{Up}(X)$ closed under finite unions and finite intersections (in particular, $\mathcal{R}$  contains $\emptyset$ and $X$).
        \item We call $\mathcal{R}$ a \textit{Priestley ring} if whenever $x,y\in X$ and $x\nleq y$, there is $A\in \mathcal{R}$ with $x\in A$ and $y\notin A$.
        \item If $X$ is an ordered space, then a Priestley ring $\mathcal{R}$ is a \textit{Priestley basis} if $\{U \,\setminus\, V \mid U,V\in \mathcal{R}\}$ is a basis for the topology on $X$. Let $\mathfrak{PB}(X)$ be the set of Priestley bases for $X$.  
    \end{enumerate}
\end{definition}

The following generalizes the corresponding result for Stone compactifications of a zero-dimensi\-onal space \cite[Thm.~13.1]{dwinger1961} (see also \cite{Magill1968}):

\begin{theorem}\label{thm: Isomorphism between Priestley bases and Priestley compactifications}
    {\em \cite[Thm.~5.2]{BezhanishviliMorandi2010}}
    For an order-zero-dimensional space $X$, the poset $(\mathcal{K}_{0}(X),\preceq)$ is isomorphic to $(\mathfrak{PB}(X),\subseteq)$.
\end{theorem}

\begin{remark}\label{Isomorphism of Priestley bases}
   The isomorphism of Theorem \ref{thm: Isomorphism between Priestley bases and Priestley compactifications} associates with each Priestley order-compactificati\-on $(Y,e)$ of $X$ the Priestley basis $$\mathcal{R}_{Y}:=\{ e^{-1}(U) : U\in \mathsf{ClopUp}(Y)\},$$
   and with each Priestley basis $\mathcal{R}$ of $X$, the Priestley order-compactification $(\mathsf{Spec}(\mathcal{R}),e)$, where $\mathsf{Spec}(\mathcal{R})$ is the Priestley space of $\mathcal{R}$ and $e : X \to \mathsf{Spec}(\mathcal{R})$ is given by $$e(x) = \{ U \in \mathcal{R} : x \in U \}.$$
\end{remark}

For a Priestley order-compactification $(Y,e)$ of $X$, let 
$$\mathcal{B}_{Y}:=\{ e^{-1}(W) : W\in \mathsf{Clop}(Y)\}.$$
Then $\mathcal{B}_{Y}$ is a Boolean basis of $X$ (that is, $\mathcal{B}_Y$ is a Boolean subalgebra of $\mathsf{Clop}(X)$ and a basis for the topology on $X$).

\begin{definition}\label{Defn: N-bases}
Let $X$ be order-zero-dimensional and $(Y,e)$ a Priestley order-compactification of $X$. We say that the associated Priestley basis $\mathcal{R}_{Y}$ is an \textit{N-basis} provided for all $W,V\in \mathcal{B}_{Y}$, $${\uparrow}W\cap {\downarrow}V=\varnothing \Longrightarrow \exists K\in \mathcal{R}_{Y} : W\subseteq K \mbox{ and } V\subseteq X \,\setminus\, K.$$
\end{definition}

The next theorem characterizes which Priestley order-compacti\-fi\-cations are N-order-compacti\-fi\-cations.

%\begin{lemma}
%Let $X$ be a space, $Y$ an order-zero-dimensional order-compactification and $R_{Y}$ the Priestley base associated to this order-compactification. If $W\in B_{Y}$, then $cl_{Y}(W)$ is clopen. Moreover if $U\in R_{Y}$, then $cl_{Y}(U)$ is a clopen upset.
%\end{lemma}
%\begin{proof}
%Immediate from the definitions and the fact that $X$ is dense.
%\end{proof}

\begin{theorem}\label{Prop: N-order-compactifications and N-bases}
Let $X$ be order-zero-dimensional and $(Y,e)$ a Priestley order-compactification of $X$. Then $(Y,e)$ is an N-order-compactification iff $\mathcal{R}_{Y}$ is an N-base.
\end{theorem}

\begin{proof}
To simplify notation, we assume that $X$ is a subspace of $Y$ and $e$ is the identity (see Remark~\ref{rem: identification}).
First suppose that $Y$ is an N-order-compactification of $X$. Then $\le_Y$ is the closure of $\le_X$ (in $Y^2$). Let $W,V$ be clopens in $Y$ such that ${\uparrow}_X(W\cap X)\cap {\downarrow}_X(V\cap X)=\emptyset$, where ${\uparrow}_{X}(W \cap X)$ denotes the upset of $W \cap X$ in $X$ and ${\downarrow}_X(V\cap X)$ the downset of $V\cap X$ in $X$. We show that ${\uparrow}W \cap {\downarrow}V=\varnothing$ in $Y$. By assumption, ${\uparrow}_X(W\cap X)\cap {\downarrow}_X(V\cap X)=\emptyset$, so $\leq_{X}\cap (W\times V) = \emptyset$. Since $Y$ is an N-order-compactification, %of $X$, 
$\le_Y$ is the closure of $\le_X$. Therefore, $\le_Y \cap (W\times V) = \emptyset$, and hence ${\uparrow}W \cap {\downarrow}V=\varnothing$.
%\color{red} For suppose not; then we will have $y\leq_{Y} x$ where $y\in W$, and $x\leq_{Y} z$ where $z\in V$, then $y\leq_{Y} z$. Because $Y$ is an $N$-order-compactification, and so $\leq_{Y}$ is the closure of $\leq_{X}$, we have $\leq_{X}\cap (W\times V)\neq \emptyset$. Hence there are $y'\in W\cap X$ and $z'\in V\cap X$ such that $y'\leq_{X}z'$, a contradiction to our assumption that ${\uparrow}_{X}(W\cap X)\cap {\downarrow}_{X}(V\cap X)=\emptyset$ \color{black}. Therefore, ${\uparrow}W \cap {\downarrow}V=\varnothing$.
Since $Y$ is a Priestley space, by Fact \ref{fact: Basic facts about Priestley spaces}.\ref{eq: Closed upsets can be separated} we can find a clopen upset $U$ of $Y$ such that $W\subseteq U$ and $V\cap U = \varnothing$. Letting $K = U\cap X$ completes the proof 
%witnesses the fact 
that $\mathcal{R}_{Y}$ is an N-basis.

Conversely, suppose that $Y$ is not an N-order-compactification of $X$. Then there is $(x,y)$ in $\le_Y$ such that $(x,y)$ is not in the closure of $\le_X$. Therefore, 
%a pair $x\leq y$, and 
there are clopens $W,V$ of $Y$ such that $(x,y) \in W \times V$ but $(W\times V)\cap \le_X = \varnothing$. 
%Letting ${\uparrow}_{X}(W \cap X)$  denote the upset of $W \cap X$ in $X$ and ${\downarrow}_X(V\cap X)$ the downset of $V\cap X$ in $X$, 
%\color{red} For $W\subseteq Y$, let ${\uparrow}_{X}W={\uparrow}W\cap X$. \color{black} Then 
Thus, ${\uparrow}_X(W\cap X) \cap {\downarrow}_X(V\cap X) = \emptyset$. We show that there is no $K\in \mathcal{R}_{Y}$ such that $W\cap X\subseteq K$ and $(V\cap X)\cap K = \varnothing$. Suppose, for contradiction, %Assume towards a contradiction 
that such a $K$ exists. Since $K \in \mathcal{R}_Y$, there is a clopen upset $U$ of $Y$ such that $K = U \cap X$. Because $X$ is dense in $Y$, we have  $U=\overline{U\cap X}=\overline{K}$. This implies that $W\cap X\subseteq U$, so $W=\overline{W\cap U}\subseteq U$. Also, $(V\cap X)\cap K=\emptyset$, so $(V\cap X)\cap U=\varnothing$, and hence $V \cap U = \overline{V\cap X}\cap U = \varnothing$. 
Since $x\in U$ and $U$ is an upset of $Y$, we have $y\in U$. But $y\in V$, so $V\cap U \ne \varnothing$, a contradiction.
\end{proof}

\begin{remark}
    %The result of 
    Theorem \ref{Prop: N-order-compactifications and N-bases} is reminiscent of 
%the characterization given by Blatter 
\cite[Thm.~1.6]{BLATTER197556}, which characterizes N-order-compacti\-fi\-cations amongst all order-compactifications of $X$. %Nevertheless, the separability condition in Definition \ref{Defn: N-bases} is new as far as we are aware.
\end{remark}

\color{black}

For an order-zero-dimensional space $X$, there is clearly the largest Priestley basis of $X$, namely the basis $\mathsf{ClopUp}(X)$ of all clopen upsets of $X$. Consequently, there is the largest Priestley order-compactification of $X$, which we denote by $\eta_{0}X$. 
It can be constructed by taking the Priestley space of $\mathsf{ClopUp}(X)$. 

As the notation suggests, $\eta_{0}X$ provides a generalization of 
%It is the analogue of 
the Banaschewski compactification $\beta_0X$ of a zero-dimensional space $X$ \cite{Banaschewski1955} (see also \cite{Banaschewski1963}).
It is a classic result 
%\cite{Banaschewski1955} 
that the Banaschewski compactification of a zero-dimensional space $X$ coincides with the Stone-\v{C}ech compactification of $X$ iff $X$ is strongly zero-dimensional. 
We have a similar 
situation for when $\eta_{0}X$ coincides with $\eta X$, 
for which we need to strengthen the notion of order-zero-dimensionality:

\begin{definition}
    Let $X$ be an ordered 
    %topological 
    space.
    \begin{enumerate}
       \item  (Complete-order-separation) Let $A\subseteq X$ be an upset, $B\subseteq X$ a downset, and ${A\cap B=\emptyset}$. We say that $A$ and $B$ are \textit{completely-order-separated} if there exists a continuous order-preserving %function 
       $f:X\to [0,1]$ such that $A\subseteq f^{-1}(1)$ and $B\subseteq f^{-1}(0)$.
        
        \item (Strong Priestley separation) We say that $X$ satisfies the \textit{strong Priestley separation axiom} if whenever $A,B$ are completely-order-separated, 
        %then 
        there is a clopen upset $U$ such that $A\subseteq U$ and $B\subseteq X \,\setminus\, U$.  
        \item (Strong-order-zero-dimensionality) We call $X$  \textit{strongly-order-zero-dimensional} if it satisfies the strong Priestley separation axiom.
    \end{enumerate}
\end{definition}

%It should be noted that 
For completely-order-regular spaces, strong-order-zero-dimensionality implies order-zero-dimen\-sionality \cite[Prop.~4.2]{BezhanishviliMorandi2010}. Since every compact ordered space is completely-order-regular (see \cite[p.~46]{nachbin1965topology}),  
we arrive at the following generalization of Dwinger's result \cite[Thm.~10.2]{dwinger1961}:

\begin{proposition}
    Let $X$ be a Nachbin space. The following are equivalent.
    \begin{enumerate}
        \item $X$ is a Priestley space;
        \item $X$ is order-zero-dimensional;
        \item $X$ is strongly-order-zero-dimensional.
    \end{enumerate}
\end{proposition}

\begin{proof}
    Since $X$ is a Nachbin space, we have (3)$\Rightarrow$(2)$\Rightarrow$(1). 
    %As noted above, (3) implies (2) 
    %since the space is compact ordered. 
    %and it is obvious that (2) implies (1).
    %by definition. 
    To see that (1)$\Rightarrow$(3), %implies (3), 
    let $A,B$ be completely-order-separated. 
    %note that $A\subseteq X$ is an upset and $B\subseteq X$ is a downset, and they are completely ordered separated, 
    Then ${\uparrow}{\sf cl}(A)\cap {\downarrow}{\sf cl}(B)=\emptyset$.
    Therefore, there is a clopen upset $U$ such that $A\subseteq U$ and $B\subseteq X \,\setminus\, U$ (see Fact \ref{fact: Basic facts about Priestley spaces}.\ref{eq: Closed upsets can be separated}), as desired.
\end{proof}

\begin{proposition}\label{Strongly Zero Dimensional if and only if Nachbin is Priestley}
    Let $X$ be a completely-order-regular space. The following are equivalent.
    \begin{enumerate}
        \item $X$ is strongly-order-zero-dimensional;
        \item $\eta X$ is a Priestley space;
        \item $\eta_{0}X\cong \eta X$.
    \end{enumerate}
\end{proposition}

\begin{proof}
    (1)$\Leftrightarrow$(2) This is shown in  \cite[Thm.~2.9]{Nailana2000} (see also  \cite[Prop.~4.4]{BezhanishviliMorandi2010}).

    (2)$\Rightarrow$(3) If $\eta X$ is a Priestley order-compactification, then 
    %By Theorem \ref{Size of Priestley order-compactifications} we have that then 
    $\eta X\preceq \eta_{0}X$, and hence 
    %but certainly also 
    $\eta X\cong\eta_{0}X$.
    
    (3)$\Rightarrow$(1) If $\eta X\cong\eta_{0}X$, then clearly $\eta X$ is a Priestley space since so is $\eta_{0}X$.
\end{proof}

\section{Heyting and Esakia order-compactifications}\label{Section: Esakia order-compactifications}

In this section we develop the theory of 
%will be interested in 
those Priestley order-compactifications that result in an Esakia space. Since Esakia spaces serve as Priestley duals of Heyting algebras \cite{esakiatopologicalkripke}, such order-compactifications are of importance in the study of intuitionistic logic and related systems. 
%A key aspect of interest in such spaces concerns the fact that the relation, in addition to being closed, is \textit{continuous}:

\begin{definition} \label{def: continuously ordered}
    Let $X$ be an ordered space. We say that 
    \begin{enumerate}
        \item the order $\leq$ on $X$ is 
        %\textit{continuously ordered space}, or that $\leq$ is 
        \textit{continuous} provided
        \begin{enumerate}
        \item[(i)] ${\uparrow}x$ is closed for each $x \in X$ and (ii)\,~$
        U \mbox{ open } \Longrightarrow  {\downarrow}U \mbox{ is open};
        $ 
        \end{enumerate}
        %if whenever 
        \item $X$ is {\em continuously ordered} provided $\leq$ is a continuous order;
        \item $X$ is an \textit{Esakia space} provided $X$ is a continuously ordered Stone space.
    \end{enumerate}  
\end{definition}

Esakia spaces were introduced by Esakia \cite{esakiatopologicalkripke} under the name of ``hybrids" (of topology and order) to provide a duality theory for Heyting algebras. They have been extensively studied since and are now widely known as Esakia spaces. %\color{blue} (add some REFs). \color{black}
%\begin{remark}
    There are several equivalent characterizations for an ordered space $X$ 
    %-Hausdorff space $(X,\leq)$ 
    to be Esakia. We only list the following two and refer the interested reader to \cite[Thm.~3.1.2]{Esakiach2019HeyAlg} for more.

    \begin{theorem}
        For an ordered Stone space $X$, the following are equivalent.
        \begin{enumerate}
        \item $X$ is Esakia.
        \item $X$ is a Priestley space and $U$ open $\Longrightarrow {\downarrow}U$ is open.
        %for each $C\subseteq X$ a clopen, ${\downarrow}C$ is clopen;
        \item 
        %$X$ is a Priestley space 
        ${\downarrow}x$ is closed for each $x \in X$ and  ${\uparrow}\overline{A} = \overline{{\uparrow}A}$ for each $A \subseteq X$.
    \end{enumerate}
    \end{theorem}
%\end{remark}

By Esakia duality, Heyting homomorphisms between Heyting algebras are dually characterized by those continuous order-preserving maps between their Esakia spaces that satisfy the following 
additional condition:

\begin{definition}
    An order-preserving map $f:X\to Y$ between posets is a \textit{p-morphism} if for all $x\in X$ and $y\in Y$, from $f(x)\leq y$ it follows that there is $x'\geq x$ with $f(x')=y$ (see Figure \ref{fig:pmorphismdiagram}).
\begin{figure}[h]
    \centering
\begin{tikzcd}
x' \arrow[r, "f", dashed]                   & y                       \\
x \arrow[r, "f"'] \arrow[u, "\leq", dashed] & f(x) \arrow[u, "\leq"']
\end{tikzcd}    \caption{}
    \label{fig:pmorphismdiagram}
\end{figure}
\end{definition}

\begin{remark}\label{rem: Different presentations of p-morphisms}\
    \begin{enumerate}
        \item \label{eq: different p-morphisms} It is well known (see, e.g., \cite[Prop.~1.4.12]{Esakiach2019HeyAlg}) that for an order-preserving map $f:X\to Y$, the following are equivalent:
        \begin{itemize}
        \item $f$ is a p-morphism.
        \item If $U$ is an upset of $X$, then $f[U]$ is an upset of $Y$.
            \item For each $x\in X$, ${\uparrow}f(x)\subseteq f[{\uparrow}x]$;
            \item For each $y\in Y$, $f^{-1}[{\downarrow}y]\subseteq {\downarrow}f^{-1}[y]$.
        \end{itemize}
        \item It is worth pointing out that while order-homeomorphisms are necessarily p-morphisms, not every order-embedding is a p-moprhism.
    \end{enumerate}
\end{remark}

%\color{red}
%\begin{definition}\
%A \textit{Heyting algebra} is a 
%   lattice $H$ with an additional binary operation $\to:H^2\to H$ such that, for all $a,b,c\in H$,
%   \begin{equation*}
%       a\wedge c\leq b\iff c\leq a\rightarrow b.
%   \end{equation*}
%\end{definition}

%It is well known that each Heyting algebra is a distributive lattice.

%\begin{definition}
%    A lattice homomorphism $f:H\to H$ between Heyting algebras is a \textit{Heyting homomorphism} if $f(a\rightarrow b)=f(a)\rightarrow f(b)$.
%\end{definition} 

%We have that Priestley representation specializes to Esakia representation of Heyting algebras \cite{esakiatopologicalkripke}:

%\color{black}

\begin{theorem} [Esakia duality, \cite{esakiatopologicalkripke}]
The category $\sf Esa$ of Esakia spaces and continuous p-morpshisms is dually equivalent to the category $\sf HA$ of Heyting algebras and Heyting 
%algebra 
homomorphisms. 
\end{theorem}

Esakia duality is a restricted version of Priestley duality (note that neither $\sf Esa$ is a full subcategory of $\sf Pries$ nor $\sf HA$ is a full subcategory of $\sf DLat$). We will freely use the following well-known facts about Esakia spaces (see, e.g., 
%\cite[Lem.~3.3.13]{Esakiach2019HeyAlg}):
\cite{Esakiach2019HeyAlg}):

\begin{fact}\ \label{fact: Condition for being a continuous p-morphism} 
\begin{enumerate}
   \item \label{eqref: Each closed upset is an Esakia space} Each closed upset of an Esakia space is an Esakia space.
   \item \label{eqref: formula for implication} If $X$ is an Esakia space, then the lattice $\mathsf{ClopUp}(X)$ is a Heyting algebra where the Heyting implication is calculated by 
\begin{equation*}
    U\rightarrow_{\mathsf{ClopUp}(X)} V \coloneqq X\,\setminus\,{\downarrow}(U\,\setminus\,V)=\{x\in X : {\uparrow}x \cap U \subseteq V \}
\end{equation*}
for each $U,V\in\mathsf{ClopUp}(X)$.
\item \label{eqref: Cont p-morphisms} Let $X,Y$ be continuously ordered spaces, so that $\mathsf{ClopUp}(X)$ and $\mathsf{ClopUp}(Y)$ are Heyting algebras. Then a continuous order-preserving map $f:X\to Y$ is a p-morphism iff ${f^{-1}:\mathsf{ClopUp}(Y) \to \mathsf{ClopUp}(X)}$ is a Heyting %algebra 
homomorphism, which amounts to $$f^{-1}{\downarrow}(U\,\setminus\,V)\subseteq {\downarrow}f^{-1}[U\,\setminus\,V]$$ 
%if and only if 
for each $U,V\in\mathsf{ClopUp}(Y)$.
   \item \label{eqref: Heyting algebra one to one if p-morphism onto} A Heyting %algebra 
   homomorphism $f:H\to H'$ is one-to-one iff $f^{-1}:\mathsf{Spec}(H')\to\mathsf{Spec}(H)$ is onto.
\end{enumerate}
\end{fact}

%We can distinguishing the two by calling the weaker notion a Heyting ring (a subring of the upsets that happens to be a Heyting algebra) and Esakia ring (a Heyting subalgebra of the upsets). The two order-compactifications can also be termed Heyting order-compactification and Esakia order-compactification.
%\color{black}

Let $X$ be order-zero-dimensional, $Y$ a Priestley order-compactification of $X$, and $\mathcal{R}_Y$ the corresponding Priestley basis of $X$. By Esakia duality, $Y$ is an Esakia space iff $\mathcal{R}_Y$ is a Heyting algebra. However, $\mathcal{R}_Y$ may not be a Heyting subalgebra of $\mathsf{ClopUp}(X)$, as we next see.  

\begin{example} \label{example: HA but not H subalgebra}
    The set $X$ of natural numbers with the trivial order and discrete topology is clearly order-zero-dimensional. We let $Y := X \cup\{\infty\}$ be the one-point compactification of $X$, and we order $Y$ by setting $x\leq y$ iff $x=y$ or $y=\infty$ (see Figure \ref{fig:Heytingalgebranotsubalgebra}).

\begin{figure}[h]
    \centering
\begin{tikzpicture}
    \node at (0,0) {$\bullet$};
    \node at (0,-0.5) {$0$};
     \node at (1,0) {$\bullet$};
         \node at (1,-0.5) {$1$};

      \node at (2,0) {$\bullet$};
               \node at (2,-0.5) {$2$};

       \node at (3,0) {$\bullet$};
            \node at (3,-0.5) {$3$};

        \node at (4,0) {$\dots$};

        \node at (2,1) {$\bullet$};
        \node at (2,1.3) {$\infty$};
        \node at (2,-1.2) {$Y$};

        \draw (0,0) -- (2,1) -- (1,0) -- (2,1) -- (2,0) -- (2,1) -- (3,0);
\end{tikzpicture}
\caption{}
\label{fig:Heytingalgebranotsubalgebra}
\end{figure}
    
    It is well known (see, e.g., \cite[Thm.~3.2.4]{Esakiach2019HeyAlg}) that  $Y$ is an Esakia space, so $\mathcal{R}_Y$ is a Heyting algebra. However, $\mathsf{ClopUp}(X)$ is the powerset $\wp(X)$, and clearly $\mathcal{R}_Y$ is not a Heyting subalgebra of $\wp(X)$ (for example, if $U = X \,\setminus\, \{0\}$, then $U \to_{\mathcal{R}_Y} \varnothing = \varnothing$ while $U \to_{\wp(X)} \varnothing = \{0\}$). 
\end{example}

This gives rise to the following two notions of a ring of upsets of a poset.

%We likewise have a notion of a ring of sets corresponding to Heyting algebras. Here we notice a key difference in the theory:

\begin{definition}
We call a Priestley ring $\mathcal{R}$ of a poset $X$ 
\begin{enumerate}
        \item a \textit{Heyting ring} provided $\mathcal{R}$ is a Heyting algebra; 
        \item an {\em Esakia ring} provided $\mathcal{R}$ is a Heyting subalgebra of $\mathsf{Up}(X)$.
    \end{enumerate} 
\end{definition}

This, in turn, gives rise to the following two notions of bases of an order-zero-dimensional space.

\begin{definition}
We call a Priestley basis $\mathcal{R}$ of an order-zero-dimensional space $X$ 
    \begin{enumerate}
        \item a {\em Heyting basis} provided $\mathcal{R}$ is a Heyting ring;
        \item an {\em Esakia basis} provided $\mathcal{R}$ is an Esakia ring.
    \end{enumerate}
\end{definition}

Let $(\mathfrak{EB}_{H}(X),\subseteq)$ be the poset of Heyting bases and $(\mathfrak{EB}_{E}(X),\subseteq)$ the poset of Esakia bases of~$X$.
Each of these 
%above two bases 
give rise to special Priestley order-compactifications we describe next.

\begin{definition}\label{def: Heyting and Esakia compactifications}
    Let $X$ be order-zero-dimensional. We call a Priestley order-compactification $Y$ of $X$ 
    \begin{enumerate}
        \item a \textit{Heyting order-compactification} if $Y$ is an Esakia space; 
        \item an {\em Esakia order-compactification} if it is a Heyting order-compactification and ${\uparrow}_{X}x$ is dense in ${\uparrow}_{Y}x$ for each $x\in X$.
    \end{enumerate}
\end{definition}

\begin{theorem}\label{Theorem: Heyting and Esakia compactifications characterization}
    Let $X$ be order-zero-dimensional. A Priestley order-compactification $Y$ of $X$ is 
    %space $X$, we have:
    \begin{enumerate}
        \item \label{eq: Heyting comp and Heyting basis} a Heyting order-compactification iff $\mathcal{R}_Y$ is a Heyting basis;
        \item \label{eqref: Priestley compactification and Esakia bases} an Esakia order-compactification iff $\mathcal{R}_Y$ is an Esakia basis.
    \end{enumerate}
\end{theorem}

\begin{proof}
    \eqref{eq: Heyting comp and Heyting basis} By Esakia duality, $Y$ is a Heyting order-compactification iff $\mathsf{ClopUp}(Y)$ is a Heyting algebra. Since $\mathcal{R}_{Y}$ is isomorphic to $\mathsf{ClopUp}(Y)$, 
%the latter 
this is equivalent to $\mathcal{R}_{Y}$ being a Heyting basis. 
%Conversely, if $\mathcal{R}_Y$ is a Heyting basis, then $\mathcal{R}_Y$ is a Heyting algebra. Therefore, $\mathsf{ClopUp}(Y)$ is a Heyting algebra, and hence $Y$ is an Esakia space. Thus, $Y$ a Heyting order-compactification.

  \eqref{eqref: Priestley compactification and Esakia bases} 
%To establish this, 
%We first show that 
%a space 
%$Y$ is an Esakia order-compactification of $X$ %if and only if 
%iff $\mathcal{R}_{Y}$ is an Esakia basis. \color{blue} This result might deserve to stand on its own. \color{black}
%\rodrigonote{In the previous round we had this result as its own theorem. Shall we agree to isolate it? If so, where shall we put it? And in that case, the proof of Theorem \ref{Theorem: Isomorphism between Esakia compactifications and Esakia bases} will be quite short, but hopefully that is alright.} 
First suppose that $Y$ is an Esakia order-compactification of $X$. Let $E,F \in \mathcal{R}_Y$. Then there are $U,V\in\mathsf{ClopUp}(Y)$ such that $E=U\cap X$ and $F=V\cap X$. 
%It sufficies to 
We show that 
\begin{equation*}
    E\rightarrow_{\mathsf{Up}(X)} F =(U\rightarrow_{\mathsf{ClopUp}(Y)} V)\cap X.
\end{equation*}
%$\leq_{X}=(X\times X)\cap \leq$. Indeed, we need to show that if $U,V$ are clopen upsets in $Y$, and consequently, $U\cap X,V\cap X\in R_{Y}$, then $(U\cap X)\rightarrow_{\mathsf{ClopUp}(X)} (V\cap X)\in R_{Y}$. For that we show:
%\begin{equation*}
%    (U\cap X)\rightarrow_{\mathsf{ClopUp}(X)} (V\cap X))=(U\rightarrow_{R_{Y}} V)\cap X;
%\end{equation*}
One inclusion is immediate: if $x\in (U\rightarrow_{\mathsf{ClopUp}(Y)} V)\cap X$, then $x\in X$ and ${\uparrow}_Y x \cap U \subseteq V$, so ${\uparrow}_X x \cap E \subseteq F$, and hence $x\in E\rightarrow_{\mathsf{Up}(X)} F$. For the other inclusion, suppose that $x\in X$, but $x\notin U\rightarrow_{\mathsf{ClopUp}(Y)} V$. Then there is $y\geq_Y x$ such that $y\in U\,\setminus\,V$. Since $Y$ is an Esakia order-compactification of $X$, the set ${\uparrow}_Y x\cap X$ is dense in ${\uparrow}_Y x$. Therefore, there is $y'\in{\uparrow}_Y x\cap X$
%$y'\geq x$ 
with $y'\in U\,\setminus\,V$. Thus, $x\notin (U\cap X)\rightarrow_{\mathsf{Up}(X)} (V\cap X)$. This shows that $\mathcal{R}_{Y}$ is a Heyting subalgebra of $\mathsf{Up}(X)$, and hence $\mathcal{R}_{Y}$ is an Esakia basis of $X$.

Conversely, suppose that $\mathcal{R}_Y$ is an Esakia basis of $X$. Recalling that 
%we denote by 
$\phi(A)=\{x\in Y: A\in x\}$ for each $A\in\mathcal{R}_Y$, we have:
%. We first show the following:

\begin{claim}\label{Claim: Identity of elements in Esakia basis}
%\begin{equation*}
$\phi(E\rightarrow_{\mathcal{R}_Y}F)=\phi(E)\rightarrow_{\mathsf{ClopUp}(Y)}\phi(F)$
%\end{equation*}
for each $E,F\in \mathcal{R}_Y$.
\end{claim}
\begin{proof}
    The inclusion $\phi(E\rightarrow_{\mathcal{R}_Y}F)\subseteq \phi(E)\rightarrow_{\mathsf{ClopUp}(Y)}\phi(F)$ is clear. For the other inclusion, suppose that $y\notin \phi(E\rightarrow_{\mathcal{R}_Y}F)$.
    %, so $U\rightarrow_{\mathcal{R}}V\notin y$. 
    %for $x\in Y$; then we want to show that:
    Consider the family
\begin{equation*}
    \mathcal{K}=\{ \phi(A) : A \in \mathcal{R}_Y, y\in \phi(A)\}\cup \{\phi(E)\,\setminus\,\phi(F)\}.
\end{equation*}
If $\mathcal{K}$ does not have the finite intersection property, then there is $A\in \mathcal{R}_Y$ with $y\in\phi(A)$ and $\phi(A)\cap \phi(E) \subseteq \phi(F)$, so $\phi(A\cap E)\subseteq\phi(F)$, and hence $A\cap E\subseteq F$. Since $\mathcal{R}_Y$ is an Esakia basis, $A\subseteq 
%X\,\setminus\,{\downarrow}(U\,\setminus\,V)=
E\rightarrow_{\mathcal{R}_Y}F$, so 
%(where the equality follows from $\mathcal{R}$ being an Esakia basis); since 
$y\in \phi(A)$ implies that $y\in \phi(E\rightarrow_{\mathcal{R}_Y}F)$, a contradiction. Thus, $\mathcal{K}$ has the finite intersection property.
%; By Theorem \ref{thm: Isomorphism between Priestley bases and Priestley compactifications}, $Y=\mathsf{Spec}(\mathcal{R}_{Y})$, so by 
Since $Y$ is compact, $\bigcap\{\varphi(A) : A\in\mathcal{R}_Y, y\in \varphi(A) \} \cap (\phi(E)\,\setminus\,\phi(F)) %\varphi(U)\setminus\varphi(V)
\neq \emptyset$. For each $z$ in the intersection, we have $y\leq z$
%, since $\mathcal{R}$ is a Priestley ring, 
and $z\in \varphi(E)\,\setminus\, \varphi(F)$, so ${\uparrow}y\cap\phi(E)\not\subseteq\phi(F)$, and hence $y\notin \phi(E)\rightarrow_{\mathsf{ClopUp}(Y)}\phi(F)$.
\end{proof}

Since $\mathcal{R}_Y$ is a Heyting algebra, $Y$ is an Esakia space. Therefore, to prove that 
%We now show that 
$Y$
%$(\mathsf{Spec}(\mathcal{R}),e)$ 
is an Esakia order-compactification of $X$, it  
it is sufficient to show that ${\uparrow}_Xx$
%${\uparrow}x\cap X$ 
is dense in ${\uparrow}_{Y}x$ for each $x\in X$. %Note that 
Under our identification of $X$ with its image in $Y$, we identify $x\in X$ with $e(x)=\{A\in \mathcal{R}_Y : x\in A\}$. Thus, 
%so formally we will verify 
we must verify that ${\uparrow}_Ye(x)\cap e[X]$ is dense in ${\uparrow}_{Y}e(x)$. 
%Indeed, 
Let ${\uparrow}_{Y}e(x)\cap W\ne\varnothing$
for some clopen $W\subseteq Y$. 
%Since $W$ is a finite union of the clopen upsets of the form $\phi(U)\,\setminus\,\phi(V)$, 
By Fact~\ref{fact: Basic facts about Priestley spaces}.\ref{eqref: Clopen upsets and downsets form a subbasis}, there are $E,F \in \mathcal{R}_Y$ such that  
${\uparrow}_{Y}e(x)\cap \left(\phi(E)\,\setminus\,\phi(F)\right)\neq \emptyset$. Therefore, ${\uparrow}_{Y}e(x)\cap \phi(E) \not\subseteq \phi(F)$, and hence $e(x)\notin 
%\mathsf{Spec}(\mathcal{R})\,\setminus\,{\downarrow}(\phi(U)-\phi(V))=
\phi(E)\rightarrow_{\mathsf{ClopUp}(Y)} \phi(F)$. By Claim~\ref{Claim: Identity of elements in Esakia basis}, $\phi(E)\rightarrow_{\mathsf{ClopUp}(Y)} \phi(F)=\phi(E\rightarrow_{\mathcal{R}_Y} F)$, and so $e(x)\notin \phi(E\rightarrow_{\mathcal{R}} F)$, which by definition 
implies that $x\notin E\rightarrow_{\mathcal{R}_{Y}}F$, so ${\uparrow}_Xx\cap E\not\subseteq F$. Therefore,
%From this we conclude that 
there is $y\in X$ such that $x\leq_X y$ and $y\in E\,\setminus\,F$, so $e(y)\in \varphi(E)\,\setminus\,\varphi(F)$, and hence ${\uparrow}_Ye(x)\cap e[X]\cap W\neq \emptyset$. Consequently, $Y$
%$(\mathsf{Spec}(\mathcal{R}),e)$ 
is an Esakia order-compactification of $X$.  
\end{proof}

Given two Heyting order-compactifications $(Y_{1},e_{1})$ and $(Y_{2},e_{2})$, define
        $(Y_{1},e_{1})\preceq_H (Y_{2},e_{2})$
    provided there is a continuous order-preserving map $f:Y_{2}\to Y_{1}$ such that $f \circ e_{2} = e_{1}$.
As with order-compactifications, $(Y_{1},e_{1})\preceq_H (Y_{2},e_{2})$ and $(Y_{2},e_{2})\preceq_H (Y_{1},e_{1})$ iff there is an order-homeomorphism $f:Y_1\to Y_2$ such that $f \circ e_{1} = e_{2}$, and this defines a partial order on the collection of equivalence classes of Heyting order-compactifications. We denote this poset by $(\mathcal{K}_{H}(X),\preceq_H)$, and note that it is a subposet of the poset $(\mathcal{K}_{0}(X),\preceq)$ of Priestley order-compactifications of $X$.

We also let $(\mathcal{K}_{E}(X),\preceq_E)$ be the subposet of $(\mathcal{K}_{H}(X),\preceq_H)$ consisting of Esakia order-compacti\-fications. In other words, $\preceq_E$ is the restriction of $\preceq_H$ to $\mathcal{K}_{E}(X)$. Interestingly, the order $\preceq_E$ is determined by continuous p-moprhisms between Esakia order-compactifications: 

\begin{proposition} \label{prop: restriction of the order}
    Let $(Y_{1},e_{1})$ and $(Y_{2},e_{2})$ be two Esakia order-compactifications of an order-zero-dimensional space $X$. Then $(Y_{1},e_{1})\preceq_E (Y_{2},e_{2})$ iff there is a continuous p-morphism $f:Y_{2}\to Y_{1}$ such that $f \circ e_{2} = e_{1}$.
\end{proposition}

\begin{proof}
%\eqref{eq: restriction of the order} 
One implication is clear. For the other, let $(Y_1,e_1)\preceq_{E}(Y_2,e_2)$, so there is a continuous order-preserving map $f:Y_{2}\to Y_{1}$ such that $f \circ e_{2} = e_{1}$. It is sufficient to show that $f$ is a p-morphism. Let $\mathcal{R}_{Y_1},\mathcal{R}_{Y_2}$ be the corresponding Priestley bases, which are Heyting bases by Theorem~\ref{Theorem: Heyting and Esakia compactifications characterization}.\ref{eq: Dwinger for Heyting}. By 
%the definition of $\preceq_{H}$ and 
Theorem~\ref{thm: Isomorphism between Priestley bases and Priestley compactifications}, $\mathcal{R}_{Y_1} \subseteq \mathcal{R}_{Y_2}$.   
Since $(Y_{1},e_{1})$ and $(Y_{2},e_{2})$ are Esakia order-compactifications, $\mathcal{R}_{Y_1}$ and $ \mathcal{R}_{Y_2}$ are Esakia bases. Therefore, each is a Heyting subalgebra of $\mathsf{Up}(X)$. Thus, $\mathcal{R}_{Y_1} \subseteq \mathcal{R}_{Y_2}$ implies that $\mathcal{R}_{Y_1}$ is a Heyting subalgebra of $\mathcal{R}_{Y_2}$. But then $f$ is a p-moprhism by Esakia duality. 
%which together with the inclusion by Theorem~\ref{Theorem: Heyting and Esakia compactifications characterization}.\ref{eqref: Priestley compactification and Esakia bases} and \eqref{eq: Dwinger for Esakia}, 
%Thus, $(Y_1,e_1)\preceq_{E}(Y_2,e_2)$.
\end{proof}

As an immediate consequence of Theorems~\ref{thm: Isomorphism between Priestley bases and Priestley compactifications} and \ref{Theorem: Heyting and Esakia compactifications characterization}, we obtain:

\begin{theorem}\label{Theorem: Isomorphism between Esakia compactifications and Esakia bases}
    Let $X$ be order-zero-dimensional. 
    %space $X$, we have:
    \begin{enumerate}
        %\item \label{eq: Heyting comp and Heyting basis} A Priestley order-compactification $Y$ is a Heyting order-compactification iff $\mathcal{R}_Y$ is a Heyting basis.
        \item \label{eq: Dwinger for Heyting} $(\mathcal{K}_{H}(X),\preceq_H)$ is isomorphic to $(\mathfrak{EB}_{H}(X),\subseteq)$.
        %\item \label{eqref: Priestley compactification and Esakia bases} A Priestley order-compactification $Y$ is an Esakia order-compactification iff $\mathcal{R}_Y$ is an Esakia basis.
        \item \label{eq: Dwinger for Esakia} $(\mathcal{K}_{E}(X),\preceq_{E})$ is isomorphic to $(\mathfrak{EB}_{E}(X),\subseteq)$.
        %\item \label{eq: restriction of the order} \color{red}
        %For $Y,Z$ Esakia order-compactifications, $Y\preceq_{H}Z$ if and only if $Y\preceq_{E}Z$.
        %\color{black}
    \end{enumerate}
\end{theorem}

%\begin{proof}
%\eqref{eq: Dwinger for Heyting} 
%This is a straightforward restriction of %follows immediately by restricting the isomorphism of 
%Apply \eqref{eq: Heyting comp and Heyting basis} and Theorem \ref{thm: Isomorphism between Priestley bases and Priestley compactifications}.
%: if $Y$ is a Heyting order-compactification of $X$, then $\mathcal{R}_{Y}$ is a Heyting basis; conversely, if $\mathcal{R}$ is a Heyting basis, then $\mathsf{Spec}(\mathcal{R})$ is an Esakia space, and hence a Heyting order-compactification of $X$.

%\eqref{eq: Dwinger for Esakia} If $\mathcal{R},\mathcal{S}$ are two Heyting bases of $X$ with 
%$\mathcal{R}\subseteq \mathcal{S}$, then $\mathcal{R}$ is a Heyting subalgebra of $\mathcal{S}$. Consequently,
% there is an onto p-morphism $f:\mathsf{Spec}(\mathcal{S})\to \mathsf{Spec}(\mathcal{R})$ (see Fact \ref{fact: Condition for being a continuous p-morphism}.\ref{eqref: Heyting algebra one to one if p-morphism onto}), which is clearly the identity on $X$. Likewise, if there is an onto p-morphism $f:Y_{1}\to Y_{2}$ which is the identity on $X$ 
 %making the diagram commute with respect to $e_{1}:X\to Y_{1}$ and $e_{2}:X\to Y_{2}$, 
% then $\mathcal{R}_{Y_{2}}\subseteq \mathcal{R}_{Y_{1}}$, and 
 %we have that 
% these two transformations are each other's inverses. Thus, by Theorem \ref{thm: Isomorphism between Priestley bases and Priestley compactifications}, 
% $(K_{E}(X),\preceq_{E})$ is isomorphic to $(\mathfrak{EB}_{E}(X),\subseteq)$.
%\end{proof}

Note that not every Heyting order-compactification is an N-order-compactification. Indeed, in Example~\ref{example: HA but not H subalgebra}, 
%the space 
$Y$ is a Heyting order-compactification of $X$, but it is not an N-order-compactification: for any natural number $n$, we have $n\leq_Y \infty$, but 
%$(\{n\}\times (Y\,\setminus\,\{n\})) \cap \le_X = \varnothing$, so 
$(n,\infty)$ does not belong to the closure of $\le_X$. 
%isolates $\infty$. 
On the other hand, each Esakia order-compactification is indeed an N-order-compactification, as we next show.

\begin{proposition}
    Let $X$ be order-zero-dimensional. If $Y$ is an Esakia order-compactification, then $Y$ is an N-order-compactification.
\end{proposition}

\begin{proof}
    We may assume that $X$ is a dense subspace of $Y$ and $e$ is the identity.
    Let $(x,y)$ be in $\le_Y$, so $x\le_Y y$, and let $W \times V$ be a basic clopen neighborhood of $(x,y)$. It is sufficient to show that $(W\times V) \cap {\le_X} \ne \varnothing$.
    %$x,y\in Y$ and 
    %$x\leq y$ in $Y$ and $W,V$ be two clopens of $Y$ such that $x\in W$ and $y\in V$. 
    Since $x \le_Y y$, $W,V$ are clopen in $Y$,  and $\le_Y$ is continuous, we have that $W\cap {\downarrow}_Y V$ is a nonempty clopen of $Y$, so $(W\cap {\downarrow}_Y V)\cap X \ne \varnothing$. Therefore, 
    there is $x'\in X$ such that 
    $x'\in W \cap {\downarrow}_Y V$. Because $Y$ is an Esakia order-compactification, $X\cap {\uparrow}_Y x'$ is dense in ${\uparrow}_Y x'$. Thus, since ${\uparrow}_Y x'\cap V\neq \emptyset$, there is $y'\in X$ with $x'\leq_Y y'$ and $y'\in V$. As ${\le_Y} \cap X^2 = {\le_X}$, we obtain that $x' \le_X y'$, so $(W\times V) \cap {\le_X} \ne \varnothing$,
    %. This shows that $\leq_Y$ is the closure of $\le_X$, 
    and hence $Y$ is an N-order-compactification.
\end{proof}

Observe that not every order-zero-dimensional space admits a Heyting or Esakia order-compact\-ification. Indeed, if $X$ is a Priestley space which is not Esakia, then $X$ does not admit any Heyting order-compactification, and hence it also does not admit any Esakia order-compactification. This motivates the following definition.

\begin{definition} \label{def: H and E OZD}
    We call an order-zero-dimensional space $X$
    \begin{enumerate}
        \item \textit{H-order-zero-dimensional} provided 
        $X$ admits a Heyting order-compactification (meaning that ${\mathcal{K}_{H}(X)\neq \emptyset}$);
        \item {\em E-order-zero-dimensional} provided 
        $X$ admits an Esakia order-compactification (meaning that  ${\mathcal{K}_{E}(X) \neq \emptyset}$).
    \end{enumerate} 
\end{definition}

%We have the following
%can strengthen Theorem \ref{Theorem: Isomorphism between Esakia compactifications and Esakia bases} to a 
%characterization of E-order-zero-dimensionality:
We next show that admitting an Esakia order-compactification is equivalent to the order being continuous.

\begin{theorem} \label{prop: cont ordered}
    For an order-zero-dimensional space $X$, the following are equivalent.
    \begin{enumerate}
        \item $X$ is continuously ordered.
        \item $\eta_{0}X$ is an Esakia order-compactification.
        \item $X$ is E-order-zero-dimensional. 
        \item $\mathsf{ClopUp}(X)$ is an Esakia basis.
    \end{enumerate}
\end{theorem}
\begin{proof}
%(1)$\Rightarrow$(2) 
(2)$\Leftrightarrow$(4) This follows from Theorem \ref{Theorem: Isomorphism between Esakia compactifications and Esakia bases}.\ref{eqref: Priestley compactification and Esakia bases}.

(2)$\Rightarrow$(3) This is obvious. 

(1)$\Rightarrow$(4)  
%note that if $X$ is order-zero-dimensional, then whenever $C$ is clopen, ${\downarrow}C$ is likewise clopen. Hence if 
Let $E,F \in \mathsf{ClopUp}(X)$. 
%be clopen upsets of $X$. 
Then $E \,\setminus\, F$ is clopen, so ${\downarrow}(E\,\setminus\,F)$ is clopen since $X$ is continuously ordered. Therefore, $E\rightarrow_{\mathsf{Up}(X)} F \in \mathsf{ClopUp}(X)$, 
%is a clopen upset, 
and hence $\mathsf{ClopUp}(X)$ is a Heyting subalgebra of $\mathsf{Up}(X)$.

(3)$\Rightarrow$(1) Let $Y$ be an Esakia order-compactification of $X$. It suffices to show that $U$ open in $X$ implies that ${\downarrow}_{X}U$ is open in $X$. 
%Since $U$ is open in $X$, 
We have $U=W\cap X$ for some open $W \subseteq Y$. Since $Y$ is an Esakia space, 
%For that purpose, note that 
${\downarrow}_{Y}W$ is open in $Y$. We show that
\begin{equation*}
    {\downarrow}_{X}U={\downarrow}_{Y}W\cap X.
\end{equation*}
%indeed 
One inclusion is obvious. For the other, suppose that $x\in X$ and $x\in {\downarrow}_{Y}W$. Then ${\uparrow}_Yx\cap W\neq \emptyset$. Since $Y$ is an Esakia order-compactification, ${\uparrow}_Xx$ is dense in ${\uparrow}_Yx$, so ${\uparrow}_Xx\cap W\neq \emptyset$. Therefore, there is $y'\in U$ such that $x\leq y'$, yielding that $x\in {\downarrow}_{X}U$. 
%Since $Y$ is an Esakia space, and so ${\downarrow}_{Y}W$ is open, this tells us that 
Thus, ${\downarrow}_{X}U$ is open in $X$.
\end{proof}

However, H-order-zero-dimensionality is not equivalent to E-order-zero-dimensionality:
%any of the conditions in Theorem \ref{prop: cont ordered}:
%The following example shows  that  $E$-order-zero-dimensionality, or any of the other conditions:

\begin{example}\label{example: Non-continuously ordered space which admits Esakia compactification}
Let $C$
%$=(C,\leq)$ 
be the Cantor space with its usual order. We fix a dense subset $D$ of $C$, let $D'=\{x' : x\in D\}$ be a 
%disjoint 
copy of $D$, and define $X$ to be the topological sum $C\cup D'$ 
%, with the disjoint union topology, together 
with the order given by the reflexive closure of $\{(x,x') : x\in D\}$ inside $C\cup D'$. In Figure \ref{fig:counterexampletocontinuouslyordered}, the white circle denotes a point $y'\notin D'$ which is omitted in our space $C\cup D'$.

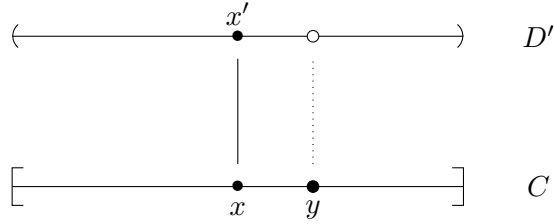
\begin{figure}[h]
    \centering
\begin{tikzpicture}
\node at (0,0) {$\bullet$};
\draw (-3,0) -- (3,0);
\node at (0,0.3) {$x'$};
\filldraw[fill=white, draw=black] (1,0) circle (0.2em);
\draw (-3,-2) -- (3,-2);
\node at (0,-2) {$\bullet$};
\node at (0,-2.3) {$x$};
\filldraw[fill=black, draw=black] (1,-2) circle (0.2em);
\node at (1,-2.3) {$y$};

\draw[dotted] (1,-1.7) -- (1,-0.3);

\draw (0,-1.7) -- (0,-0.3);

\draw (2.92,-0.15) arc (-50:50:0.2);

\draw (-2.92,-0.15) arc (230:130:0.2);

\draw (-2.85,-2.25) -- (-3,-2.25) -- (-3,-1.75) -- (-2.85,-1.75);

\draw (2.85,-2.25) -- (3,-2.25) -- (3,-1.75) -- (2.85,-1.75);

\node at (4,-2) {$C$};
\node at (4,0) {$D'$};

\end{tikzpicture}
    \caption{H-order-zero-dimensional space that is not E-order-zero-dimensional}
    \label{fig:counterexampletocontinuouslyordered}
\end{figure}

It is straightforward to check that $X$ is order-zero-dimensional. Moreover, since $D'$ is clopen in $X$, 
%the subspace $(D',\leq_{D'})$; 
if $X$ were continuously ordered,  ${\downarrow}D'=D'\cup D$ would be clopen in $X$, which would imply that $D$ is clopen in $C$, a contradiction. Thus, $X$ is not continuously ordered, and hence it is not E-order-zero-dimensional by Theorem \ref{prop: cont ordered}.

We show that $X$ is H-order-zero-dimensional. 
Let $C'$ be a copy of $C$ and let $Y$ be the topological sum $C\cup C'$ with the order given by the reflexive closure of $\{(x,x') : x\in C\}$ in $C\cup C'$.   
To see that $Y$ is an Esakia space, we clearly have that ${\uparrow}y$ is closed for each $y\in Y$. Moreover, each clopen of $Y$ is of the form $U\cup V'$, with $U,V$ clopen of $C$, so ${\downarrow}(U\cup V')=(U\cup V')\cup V$, which is clopen of $Y$. In addition, $X$ is (homeomorphic to) a dense subspace of $Y$. Thus, $Y$ is a Heyting order-compactification of $X$, and hence $X$ is H-order-zero-dimensional.  
\end{example}

\section{Locally Esakia spaces and functoriality} \label{sec: preEsakia}

It is a classic result in topology that taking the Stone-\v{C}ech compactification is left adjoint to the inclusion of the category of compact Hausdorff spaces into that of completely regular spaces (see, e.g., \cite[pp.130-131]{johnstone1982stone}).
%\cite{Cech1937}. 
Likewise, taking the Banaschewski compactification is left adjoint to the inclusion of the category of Stone spaces into that of zero-dimensional (Hausdorff) spaces \cite{Banaschewski1955}. 
%As we will see in Theorem \ref{cor: Adjunction for inclusions in Order-zero-dimensional and Priestley}, 
The latter result extends to yield that $\eta_0$ is left adjoint to the inclusion of the category of Priestley spaces into that of order-zero-dimensional spaces (see, e.g., \cite[Rem.~3.10]{BezhanishviliMorandi2010}). %\cite{Koubek1991}. 
%In this final section 
We will prove the same result for the category $\mathsf{Esa}$ of Esakia spaces. As we saw in Theorem~\ref{prop: cont ordered}, when $X$ is E-order-zero-dimensional, $\eta_0X$ is the largest Esakia order-compactification of $X$. Therefore, a natural category to work with appears to be the category of E-order-zero-dimensional spaces. However, the embedding $X \hookrightarrow \eta_0X$ may in general not be a p-morphism. We thus need to restrict our attention to those E-order-zero-dimensional spaces in which the embedding $X \hookrightarrow \eta_0X$ is indeed a p-moprhism. This gives rise to the notion of a locally Esakia space, which is the main subject of study in this section. 

%As we saw in Section \ref{Section: Esakia order-compactifications}, for arbitrary order-zero-dimensional spaces the study of those Priestley compactifications which are Esakia becomes rather involved: we have two distinct notions of Heyting and Esakia order-compactifications, and the natural notions of being zero-dimensional for this setting likewise come apart. In this section we study a subclass of order-zero-dimensional spaces which are much better behaved, and study its functoriality.

%To motivate our definitions, note that one may argue that the setting we have been working with so far is slightly unnatural -- after all, as noted, the morphisms of Esakia spaces are \textit{p-morphisms}, yet we have not required that the map from $X$ to its order-compactification $Y$ is a p-morphism -- meaning that $X$ can be identified with an upset of $Y$. This has some substantial effects:

\begin{lemma}\label{Upsets are Esakia spaces}
    Let $X$ be order-zero-dimensional and $Y$ a Heyting order-compactification of $X$. If $X$ is an upset of $Y$, then ${\uparrow}_{X}x$ is an Esakia space for each $x\in X$.
\end{lemma}

\begin{proof}
    Since $Y$ is an Esakia space, 
    %each closed upset of $Y$ is an Esakia space (see, e.g., \cite[Lem.~3.3.13]{Esakiach2019HeyAlg}). 
    %Thus, 
    ${\uparrow}_{Y}x$ is an Esakia space (see Fact~\ref{fact: Condition for being a continuous p-morphism}.\ref{eqref: Each closed upset is an Esakia space}). But since $X$ is an upset of $Y$, we have ${\uparrow}_{X}x={\uparrow}_{Y}x$, yielding the result. 
\end{proof}

In light of Lemma \ref{Upsets are Esakia spaces}, we will restrict our attention to %consider only 
those order-zero-dimensional spaces $X$ in which ${\uparrow}x$ is an Esakia space for each $x\in X$. The key to ensure this is that ${\uparrow}x$ is compact.

\begin{definition}\ \label{def: locally Esakia}
    \begin{enumerate}
    \item We call an ordered space $X$ \textit{image-compact} provided ${\uparrow}x$ is compact for each $x\in X$.
        \item We call an ordered space $X$ \textit{locally Esakia} provided it is E-order-zero-dimensional and image-compact.
        \item Let $\mathsf{LocEsa}$ denote the category of locally Esakia spaces and continuous p-morphisms.
    \end{enumerate}
\end{definition}

\begin{remark}\
\begin{enumerate}
    \item The above notion of an image-compact space generalizes the well-known notion of an {\em image-finite poset}---a poset in which ${\uparrow}x$ is finite for each $x$ (such posets were first considered %studied 
    in \cite[Def.~4.6]{Bezhanishvili2006} and were coined ``image-finite" in \cite[Def.~3.5]{BB2008}). Indeed, if the topology of an ordered space $X$ is discrete, then $X$ is image-compact iff it is image-finite. Consequently, both Esakia spaces and image-finite posets are
    %can be seen as 
    examples of image-compact spaces.
    \item Clearly $\mathsf{Esa}$ is a full subcategory of $\mathsf{LocEsa}$. In fact, a locally Esakia space is an Esakia space iff it is compact. Thus, locally Esakia spaces generalize Esakia spaces the same way locally Stone spaces generalize Stone spaces.
\end{enumerate}
    \end{remark}

%For Esakia order-compactifications $Y$ of $X$, the condition of being image-compactness turns out to be equivalent to the space $X$ being an upset in $Y$:

The next result characterizes when $X$ is an upset in its order-compactification that is an Esakia space.

\begin{proposition}\label{Lem: Equivalence of Esakia order-compactification and upset}
    Let $X$ be order-zero-dimensional and $Y$ a Heyting order-compactification of $X$. Then $X$ is an upset of $Y$ iff $X$ is image-compact and $Y$ is an Esakia order-compactification of $X$. 
    %the following are equivalent:
    %\begin{enumerate}
    %    \item $Y$ is an Esakia order-compactification and $X$ is image-compact;
    %    \item ${\uparrow}_{Y}X=X$.
    %\end{enumerate}
\end{proposition}

\begin{proof}
First suppose that $X$ is an upset of $Y$. 
%is a Heyting order-compactification and ${\uparrow}_{Y}X=X$. 
Then $X$ is image-compact by Lemma \ref{Upsets are Esakia spaces}. Moreover, since $X$ is an upset in $Y$, we have ${\uparrow}_{X}x={\uparrow}_{Y}x$ for each $x\in X$, so ${\uparrow}_{X}x$ is dense in ${\uparrow}_{Y}x$, and hence $Y$ is an Esakia order-compactification of $X$. %it is clear that the latter is dense in the former, and so the result immediately follows.

Conversely, suppose that $X$ is image-compact and $Y$ is an Esakia order-compactification of $X$. 
%is image-compact, and assume that 
Let $x\in X$ and $x \le_Y y$. 
%in $\in {\uparrow}_{Y}x$ for some  and $y\in Y$. 
%By identifying $Y$ with $\mathsf{Spec}(\mathcal{R}_Y)$, we think of $y$ as a prime filter of $\mathcal{R}_{Y}$. 
%some Esakia basis. We show that the family 
Consider the family 
    \begin{equation*}
    \mathcal{F} = \{ {\uparrow}_Y x \cap U\cap X : y \in U\} \cup \{{\uparrow}_Y x \,\setminus\, (V\cap X) : y \notin V \},
    \end{equation*}
    where $U,V\in \mathsf{ClopUp}(Y)$. If $\mathcal{F}$ does not have the finite intersection property, then  
    %Otherwise, 
    \begin{equation*}
        {\uparrow}_Y x\cap U_{1}\cap...\cap U_{n}\cap X \subseteq (V_{1}\cup...\cup V_{m}) \cap X.
    \end{equation*}
    for some $U_1,...,U_n,V_1,...,V_m\in\mathsf{ClopUp}(Y)$.
    Let $U=U_{1}\cap...\cap U_{n}$ and $V=V_{1}\cup...\cup V_{m}$. Then ${U,V\in\mathsf{ClopUp}(X)}$ and ${\uparrow}_Y x \cap U \cap X \subseteq V\cap X$. . Therefore, ${\uparrow}_Y x \cap X \subseteq (Y\,\setminus\,U)\cup V$. Since $Y$ is an Esakia order-compactification of $X$, ${\uparrow}_Xx$ is dense in ${\uparrow}_Yx$, so $${\uparrow}_Yx=\overline{{\uparrow}_Yx\cap X}\subseteq (Y\,\setminus\,U)\cup V.$$ Thus, $x\in U\rightarrow_{\mathsf{ClopUp}(Y)} V$, and hence $y\in U\rightarrow_{\mathsf{ClopUp}(Y)}V$, which is a contradiction since $y\in U$ and $y \notin V$.
    %, by upwards closure, . 
    Consequently, $\mathcal{F}$ has the finite intersection property, and because ${\uparrow}_Yx$ is compact, there is $$y'\in \bigcap \{ {\uparrow}_Yx \cap U \cap X : y\in U \} \cap \bigcap \{ {\uparrow}_Y x \,\setminus\, (V \cap X) : y\notin V \}.$$ Since $y$ and $y'$ are contained in the same clopen upsets of $Y$, it follows from the Priestley separation axiom that 
    %Then it follows that 
    $y=y'$, so 
    %by Priestley separation, meaning 
    %yielding that 
    $y\in X$, and hence $X$ is an upset of $Y$. 
\end{proof}

From this we obtain the following version
%strengthening 
of Theorem \ref{prop: cont ordered} for image-compact spaces.

\begin{theorem}\label{Special facts about Esakia spaces}
    For an order-zero-dimensional space $X$, the following are equivalent:
    \begin{enumerate}
        \item $X$ is a locally Esakia space.
        %$E$-order-zero-dimensional and image-compact.
        \item $X$ has an Esakia order-compactification $Y$ in which $X$ is an upset.
        %and ${\uparrow}_{Y}X=X$;
        \item $\eta_{0}X$ is an Esakia order-compactification in which $X$ is an upset.
        %and ${\uparrow}_{\eta_{0}(X)}X=X$.
    \end{enumerate}
\end{theorem}
\begin{proof}
    %We now show the equivalence of (1)-(3):
(3)$\Rightarrow$(2) is obvious, and (1)$\Rightarrow$(3) follows from Theorem \ref{prop: cont ordered} and Proposition  \ref{Lem: Equivalence of Esakia order-compactification and upset}. We prove (2)$\Rightarrow$(1). 
%The fact 
That $X$ is image-compact follows from Proposition \ref{Lem: Equivalence of Esakia order-compactification and upset}. We show that $X$ is continuously ordered. Let $U\subseteq X$ be open. Then $U=V\cap X$ for some open $V \subseteq Y$. 
%Note that 
Since $X$ is an upset of $Y$, $${\downarrow}_{X}U = {\downarrow}_{X}(V\cap X) = {\downarrow}_{Y}V\cap X$$
(because $x \in X$, $x \leq_Y y$, and $y\in V$ imply that $y\in V\cap X$). 
%since $X$ is an upset). 
Since $Y$ is an Esakia space, ${\downarrow}_{Y}V$ is open. Therefore, ${\downarrow}_{X}U$ is open, and hence $X$ is continuously ordered. Thus, it is left to apply Theorem~\ref{prop: cont ordered}. 
\end{proof}

Let $\mathsf{OZD}$ be the category of order-zero-dimensional spaces and continuous order-preserving maps. As we pointed out at the beginning of this section, $\eta_0$ is left adjoint to the embedding $\mathsf{Pries}\hookrightarrow\mathsf{OZD}$. Thus, each $\mathsf{OZD}$-morphism $f:X\to Z$, where $X\in\mathsf{OZD}$ and $Z\in\mathsf{Pries}$, has a unique lift $\eta_{0}f:\eta_{0}X\to Z$. We give a convenient description of the lift,  generalizing \cite[Lem.~3.1]{Bezhanishvili2006}. 

%We now study when the assignment $X \mapsto \eta_0X$ is functorial. We first look at this in the context of order-zero-dimensional spaces. Indeed, the key condition that will ensure this is the fact that one can always lift continuous order-preserving maps $f:X\to Z$ where $Z$ is a Priestley space to continuous maps $\eta_{0}(f):\eta_{0}(X)\to Z$, and the lift is unique. The following results thus generalize \cite[Lem.~3.1]{Bezhanishvili2006} directly:

\begin{proposition}\label{prop: Explicit Construction of lifts}
    Let $X\in\mathsf{OZD}$,  $Z\in\mathsf{Pries}$, and $f:X\to Z$ be continuous and order-preserving.
    \begin{enumerate}
        \item \label{eq: Formula for the lift} The unique lift $\eta_{0}f:\eta_{0}X\to Z$ is given by $(\eta_{0}f)(x)=\bigcap\mathcal{F}_{x}^{f}\cap \bigcap \mathcal{I}_{x}^{f}$ for each $x\in\eta_0X$, where
        \begin{equation*}
    \mathcal{F}^{f}_{x}=\{{\uparrow}\overline{f[U]} : U\in \mathsf{ClopUp}(X), U\in x\} \text{ and } \mathcal{I}_{x}^{f}=\{{\downarrow}\overline{f[X\,\setminus\,V]} : V\in \mathsf{ClopUp}(X), V\notin x\}.
            \end{equation*}

    \item \label{eq: Properties of Lift} For each $x\in \eta_{0}X$ and $U,V\in\mathsf{ClopUp}(Z)$, $$x\in (\eta_{0}f)^{-1}[U\,\setminus\,V] \Longleftrightarrow f^{-1}[U]\in x \mbox{ and } f^{-1}[V]\notin x.$$
    \end{enumerate}
\end{proposition}

\begin{proof}
    (\ref{eq: Formula for the lift}) Let $x\in \eta_{0}X$. First we show that 
    %\begin{equation*}
    $\bigcap\mathcal{F}_{x}^{f}\cap  \bigcap\mathcal{I}_{x}^{f}$
    %=\{y\},
    %\end{equation*}
    %for some $y\in Z$. 
    is a singleton. 
    %The fact that this intersection is non-empty follows by compactness: 
    If the intersection was empty, we would have 
    \begin{equation*}
    {\uparrow}\overline{f[U_{1}]}\cap...\cap {\uparrow}\overline{f[U_{n}]}\cap {\downarrow}\overline{f[X\,\setminus\,V_{1}]}\cap...\cap {\downarrow}\overline{f[X\,\setminus\,V_{m}]}=\emptyset
    \end{equation*}
    for some $U_{1},...,U_{n},V_{1},...,V_{m}\in\mathsf{ClopUp}(X)$. 
    %clopen upsets:
    This implies that
    %, since closure and ${\downarrow}$ are inflationary, and direct images are monotone:
    \begin{equation*}
        f[U_{1}\cap...\cap U_{n}\cap (X\,\setminus\,V_{1})\cap...\cap (X\,\setminus\,V_{k})]=\emptyset,
    \end{equation*}
    so $U_{1}\cap...\cap U_{n}\cap (X\,\setminus\,V_{1})\cap...\cap (X\,\setminus\,V_{k})=\emptyset$, and hence $U_{1}\cap...\cap U_{n} \subseteq V_{1}\cup...\cup V_{k}$, a contradiction since $U_{1}\cap...\cap U_{n} \in x$ but $V_{1}\cup...\cup V_{k} \notin x$.
    %belongs to each of these sets. 
    
    For the uniqueness, if the intersection contained $y_{1}\not\leq y_{2}$, then 
    %without loss of generality we may assume that $y_{1}\not\le y_2$, so 
    there would exist $W\in\mathsf{ClopUp}(Z)$ such that $y_1\in W$ and $y_{2}\notin W$. 
    %a clopen upset;   
    If 
    %we assume that 
    $f^{-1}[W]\in x$, then 
    %by definition, 
    $y_{2}\in {\uparrow}\overline{ff^{-1}[W]}\subseteq W$; and if 
    %we assume that 
    $f^{-1}[W]\notin x$, then $y_{1}\in {\downarrow}\overline{ff^{-1}[Z\,\setminus\,W]}\subseteq Z\,\setminus\,W$. The obtained contradiction completes the proof of uniqueness. 
    %This shows that $\eta$

    Next we prove that $\eta_0f$ is continuous. Let $W\in \mathsf{ClopUp}(Z)$. We show that
    \begin{equation*}
        (\eta_{0}f)^{-1}[W]=\bigcup \{\phi(U) : {\uparrow}\overline{f[U]}\subseteq W\}.
    \end{equation*}
    %Indeed, 
    If ${\uparrow}\overline{f[U]}\subseteq W$ and $x\in \varphi(U)$, then $(\eta_{0}f)(x)\in W$ by construction. Conversely, 
    %note that 
    if $(\eta_{0}f)(x)\in W$, then
    %\begin{equation*}
        $\bigcap_{U\in x}{\uparrow}\overline{f[U]}\subseteq W$.
    %\end{equation*}
By compactness, there are 
%finitely many 
$U_{1},...,U_{n}$ such that $U_{i}\in x$ and ${\uparrow}\overline{f[U_{i}]}\subseteq W$. Thus, letting $U=U_{1}\cap...\cap U_{n}$, we obtain ${\uparrow}\overline{f[U]}\subseteq W$ and $U\in x$, so $x\in \varphi(U)$. A similar argument shows that if $W$ is a clopen downset of $Z$, then
\begin{equation*}
    (\eta_{0}f)^{-1}[W]=\bigcup\left\{ \eta_0 X\setminus\phi(V) : {\downarrow}\overline{f[X\setminus V]}\subseteq W\right\}.
\end{equation*}
Now apply Fact \ref{fact: Basic facts about Priestley spaces}.\ref{eqref: Clopen upsets and downsets form a subbasis} to conclude that 
%Since $Z$ is a Priestley space, the clopen upsets and clopen downsets form a subbasis (\ref{fact: Basic facts about Priestley spaces}.\ref{eqref: Clopen upsets and downsets form a subbasis}), yielding that
%this shows that 
$\eta_{0}f$ is continuous.

%Now finally 
It is left to show that $\eta_0f$ is order-preserving. Let $x,y\in \eta_{0}X$ with $x\leq y$ and let 
%Then assume that 
$W\in \mathsf{ClopUp}(Z)$. If $\eta_{0}f(x)\in W$, then $x\in (\eta_{0}f)^{-1}[W]$, which is a union of clopen upsets, and hence is an upset. Therefore, $y\in (\eta_{0}f)^{-1}[W]$, and so $(\eta_0f)(y) \in W$. Thus, by the Priestley separation axiom, 
%in $Z$, we obtain that 
$\eta_{0}f(x)\leq \eta_{0}f(y)$. 

    (\ref{eq: Properties of Lift})  First, suppose that $(\eta_{0}f)(x)\in U\,\setminus\,V$. If $f^{-1}[U]\notin x$, then $(\eta_{0}f)(x)\in {\downarrow}\overline{ff^{-1}[Z\,\setminus\,U]}$ by (\ref{eq: Formula for the lift}). But 
    %certainly 
    $ff^{-1}[Z\,\setminus\,U]\subseteq Z\,\setminus\,U$, and since the latter is a closed downset, we obtain that  ${{\downarrow}\overline{ff^{-1}[Z\,\setminus\,U]}\subseteq Z\,\setminus\,U}$, implying that $(\eta_0f)(x)\notin U$, a contradiction. Therefore, $f^{-1}[U]\in x$, and a similar argument yields that 
    %one argues that 
    $f^{-1}[V]\notin x$. 
    Conversely, suppose that  
    $f^{-1}[U]\in x$ and $f^{-1}[V]\notin x$. By (\ref{eq: Formula for the lift}), 
    %note that 
    ${(\eta_{0} f)(x)\in {\uparrow}\overline{ff^{-1}[U]}\subseteq U}$ and $(\eta_{0} f)(x)\in {\downarrow}\overline{ff^{-1}[Z\,\setminus\,V]}\subseteq Z\,\setminus\,V$, so $(\eta_{0} f)(x)\in U\,\setminus\,V$.
\end{proof}

%As an immediate consequence, we obtain (see, e.g., \cite[Rem.~3.10]{BezhanishviliMorandi2010}):
%thus obtain the following result:

%\begin{theorem}\label{cor: Adjunction for inclusions in Order-zero-dimensional and Priestley}
%    The functor $\eta_{0}:\mathsf{OZD}\to \mathsf{Pries}$ is left adjoint to the inclusion of $\mathsf{Pries}$ in $\mathsf{OZD}$.
%\end{theorem}
%\begin{proof}
%Let $X$ be an order-zero-dimensional space; by Proposition \ref{prop: Explicit Construction of lifts}, we know that any continuous order-preserving map $f:X\to Y$ to a Priestley space has a unique lift $\eta_{0}f:\eta_{0}X\to Y$. Thus $\eta_{0}:\mathsf{OZD}\to \mathsf{Pries}$ is left adjoint (see, e.g., \cite[p.~91]{MacLane1978}).
%\end{proof}

%\rodrigonote{I looked a bit, but am not sure who to attribute it to, because all the papers I see do not formulate this in a particularly categorical way.}

%\color{blue} I'll try to locate the result. It might be in one of Salbany's papers. But if we can't find the above formulation anywhere, we can say that it is folklore. \color{black}

We use Proposition \ref{prop: Explicit Construction of lifts} to prove that $\eta_0$ is left adjoint to the embedding $\mathsf{Esa}\hookrightarrow\mathsf{LocEsa}$.
%obtain 
%a similar result
%We want to obtain an analogous result 
%for locally Esakia spaces, 
For this, we will make use of the following  classic result:

%al lemma\footnote{Note that Esakia's lemma can be proven with fewer assumptions, see e.g. \cite[Lem.~2.17]{BBH15}.}, due to Esakia \cite{esakiatopologicalkripke} (for a proof see \cite[Lemma 3.3.12]{Esakiach2019HeyAlg}):

\begin{proposition} (Esakia's lemma, \cite{esakiatopologicalkripke}) \label{Prop: Esakias lemma}
    Let $X$ be an Esakia space and $\mathcal{F}$ a down-directed family (with respect to inclusion) of nonempty closed subsets of $X$. Then
    \begin{equation*}
        \big{\downarrow}
        %\Big(
        \bigcap\{F: F\in \mathcal{F}\}
        %\Big)
        = \bigcap \{{\downarrow}F : F\in \mathcal{F}\}.
    \end{equation*}
\end{proposition}

\begin{remark}
Esakia's lemma can be proved in more generality; 
%with fewer assumptions, 
see 
%e.g. 
\cite[Lem.~2.17]{BBH15}.
\end{remark}
 
The next result generalizes \cite[Lem.~4.3]{Bezhanishvili2006}.

\begin{proposition}\label{Liftings are p-morphisms}
    Let $X$ be E-order-zero-dimensional. If $Z$ is an Esakia space and $f:X\to Z$ is a continuous p-morphism, then the unique lift $\eta_0 f:\eta_{0}X\to Z$ is a p-morphism.
\end{proposition}

\begin{proof}
    %To show that $\eta_{0}(p)$ is a p-morphism, we need 
    By Remark \ref{rem: Different presentations of p-morphisms}.\ref{eq: different p-morphisms}, it suffices to show that
    \begin{equation*}
        (\eta_{0} f)^{-1}[{\downarrow}y]={\downarrow}(\eta_{0}f)^{-1}[y]
    \end{equation*}
    for each $y\in Z$. 
    %Now note that 
    Since each singleton of $Z$ is the intersection of the clopen sets containing it, by Fact~\ref{fact: Basic facts about Priestley spaces}.\ref{eqref: Clopen upsets and downsets form a subbasis} and Esakia's lemma, 
    %(Proposition \ref{Prop: Esakias lemma}) 
    %and Priestley duality, \color{red} noting that each singleton is the intersection of the clopens in which it is contained, and any such clopen $W$ is a union of clopens of the form $U\,\setminus\,V$ where both $U,V$ are clopen upsets (see Fact \ref{fact: Basic facts about Priestley spaces}.\ref{eqref: Clopen upsets and downsets form a subbasis}),
    \begin{equation*}
        {\downarrow}(\eta_{0}f)^{-1}[y]=\bigcap\{{\downarrow}(\eta_{0} f)^{-1}[U\,\setminus\,V] : y\in U \text{ and } y\notin V \},
    \end{equation*}
    where $U,V$ range over $\mathsf{ClopUp}(Z)$. 
    We show that $${\downarrow}(\eta_0 f)^{-1}[U\,\setminus\,V] = (\eta_0 f)^{-1}[{\downarrow}(U\,\setminus\,V)].$$ Since $\eta_{0}f$ is order-preserving, 
    ${\downarrow}(\eta_0 f)^{-1}[U\,\setminus\,V]\subseteq (\eta_0 f)^{-1}[{\downarrow}(U\,\setminus\,V)]$. 
    For the reverse inclusion, if ${x\in (\eta_{0}f)^{-1}[{\downarrow}(U\,\setminus\,V)]}$, then $x\in %(\eta_{0}f)^{-1} {\downarrow}(U\,\setminus\,V) = 
    (\eta_{0}f)^{-1}[Z\,\setminus\,(U\rightarrow_{\mathsf{ClopUp}(Z)} V)]$ (see Fact \ref{fact: Condition for being a continuous p-morphism}.\ref{eqref: formula for implication}). Therefore, by Proposition~\ref{prop: Explicit Construction of lifts}.\ref{eq: Properties of Lift}, $f^{-1}[U\rightarrow_{\mathsf{ClopUp}(Z)} V]\notin x$.
    %, since $U\rightarrow_{\mathsf{ClopUp}(Z)}V$ is a clopen upset.  
    Since $f$ is a continuous p-moprhism, $f^{-1}$ is a Heyting 
    %algebra 
    homomorphism (see Fact \ref{fact: Condition for being a continuous p-morphism}.\ref{eqref: Cont p-morphisms}). Thus, ${f^{-1}[U]\rightarrow_{\mathsf{ClopUp}(X)} f^{-1}[V]\notin x}$, and hence $$x\notin \varphi(f^{-1}[U]\rightarrow_{\mathsf{ClopUp}(X)} f^{-1}[V])=\varphi(f^{-1}[U])\rightarrow_{\mathsf{ClopUp}(\eta_0 X)}\varphi(f^{-1}[V]),$$
    where the equality holds by Claim \ref{Claim: Identity of elements in Esakia basis} since $\mathsf{ClopUp}(X)$ is an Esakia basis.
    Therefore, there is $y\geq x$ such that $y\in \varphi(f^{-1}[U])$ and $y\notin \varphi(f^{-1}[V])$. Thus,   
    there is $y\geq x$ such that $f^{-1}[U]\in y$ and $f^{-1}[V]\notin y$. Using Proposition \ref{prop: Explicit Construction of lifts}.\ref{eq: Properties of Lift} again, $y\in (\eta_{0}f)^{-1}[U\,\setminus\,V]$, and so $x\in {\downarrow}(\eta_{0}f)^{-1}[U\,\setminus\,V]$. 
    
    Consequently, using Esakia's lemma again, 
    \begin{eqnarray*}
    {\downarrow}(\eta_{0}f)^{-1}[y] &=& \bigcap \left\{ (\eta_{0} f)^{-1}[{\downarrow}(U\,\setminus\,V)] : y\in U \text{ and } y\notin V \right\} \\ &=& (\eta_{0} f)^{-1} \bigcap \left\{ {\downarrow}(U\,\setminus\,V) : y\in U \text{ and } y\notin V \right\} \\ &=& (\eta_{0} f)^{-1} {\big\downarrow} \bigcap \left\{ U\,\setminus\,V : y\in U \text{ and } y\notin V \right\} \\ &=& 
    (\eta_{0}f)^{-1}[{\downarrow}y],
    \end{eqnarray*}
    %(see Remark \ref{rem: Condition for being a continuous p-morphism}), 
    and hence $\eta_{0}f$ is a 
    %continuous 
    p-morphism.
\end{proof}
 
% By Theorem \ref{prop: cont ordered}, if $X$ is 
% %for continuously ordered 
% E-order-zero-dimensional, then $\eta_0X$ is the top element of the poset $(\mathcal{K}_{H}(X),\preceq_H)$. If in addition $X$ is locally Esakia, then $\eta_0X$ is the top element of the poset $(\mathcal{K}_{E}(X),\preceq_E)$: 

% \begin{theorem}
%     If $X$ is a locally Esakia space, then $\eta_{0}X$ is the $\preceq_{E}$-largest Esakia order-compacti\-fication of $X$.
% \end{theorem}
% \begin{proof}
%     Let $Y$ be an Esakia order-compactification of $X$. By Proposition \ref{Lem: Equivalence of Esakia order-compactification and upset}, $X$ is an upset of $Y$. Therefore, since $X$ is E-order-zero-dimensional, by Proposition \ref{Liftings are p-morphisms} there is a p-moprhism from $\eta_0 X$ onto $Y$ that is identity on $X$. Thus, $Y\preceq_{E}\eta_{0}X$.
%     %
%     %By Theorem \ref{Special facts about Esakia spaces}, 
%     %we know that 
%     %$X$ is a locally Esakia space 
%     %image-compact and continuously ordered 
%     %iff $\eta_{0}X$ is an Esakia order-compactifica\-tion and $\varepsilon:X\to \eta_{0}X$ is a p-morphism, and every Esakia order-compactification $Y$ of $X$ has $X$ sitting inside $Y$ as an upset. 
%     %Hence, by Proposition \ref{Liftings are p-morphisms}, 
%     %we know that 
%     %whenever $Y$ is an Esakia order-compactification of $X$, $Y\preceq_{E}\eta_{0}X$. 
% \end{proof}

The following allows us to conclude that $\mathsf{Esa}$ is a reflective subcategory of $\mathsf{LocEsa}$. 

\begin{theorem}\label{Left adjoint functor to preEsakia}
    %The functor 
    $\eta_{0}:\mathsf{LocEsa}\to \mathsf{Esa}$ is left adjoint to the inclusion functor ${\mathsf{Esa}\hookrightarrow\mathsf{LocEsa}}$.
\end{theorem}
\begin{proof}
Let $X$ be a locally Esakia space. Since $X$ is image-compact, the embedding $X\hookrightarrow\eta_0X$ is a continuous p-morphism. Also, since $X$ is E-order-zero-dimensional, a continuous p-morphism $f:X\to Y$ to an Esakia space $Y$ has a unique lift $\eta_0 f : \eta_0X \to Y$ in $\mathsf{Esa}$. Thus, $\eta_{0}:\mathsf{LocEsa}\to \mathsf{Esa}$ is left adjoint to the inclusion functor $\mathsf{Esa}\hookrightarrow\mathsf{LocEsa}$ (see, e.g., \cite[p.~91]{MacLane1978}).
\end{proof}

\printbibliography[
    heading=bibintoc,
    title={Bibliography}
]

\vspace{5mm}
%\newpage

\noindent \textbf{Rodrigo Nicolau Almeida}\\
Institute for Logic, Language and Computation\\ University of Amsterdam\\
Amsterdam, 1098XH\\
The Netherlands\\
r.dacruzsilvapinadealmeida@uva.nl

\vspace{3mm}

\noindent \textbf{Guram Bezhanishvili}\\ Department of Mathematical Sciences\\
New Mexico State University\\
Las Cruces, NM 88003\\
USA\\
guram@nmsu.edu

\vspace{3mm}

\noindent \textbf{Nick Bezhanishvili}\\
Institute for Logic, Language and Computation\\ University of Amsterdam\\
Amsterdam, 1098XH\\
The Netherlands\\
n.bezhanishvili@uva.nl

%\rodrigonote{I could not get the affiliation package to work, unfortunately.}

\end{document}